\documentclass[a4paper,12pt,oneside,reqno]{amsart}
\usepackage{hyperref}
\usepackage[headinclude,DIV13]{typearea}
\areaset{15.1cm}{25.0cm}
\usepackage{txfonts,amssymb,amsmath,amsthm,bbm}

\usepackage[latin1] {inputenc}
\usepackage{graphicx, psfrag}
\usepackage{xcolor}
\usepackage{subfigure}

\newtheorem{theorem}{Theorem}[section]
\newtheorem{lemma}[theorem]{Lemma}
\newtheorem{proposition}[theorem]{Proposition}

\definecolor{bbm}{RGB}{51,153,0}
\definecolor{above}{RGB}{128,0,128}
\definecolor{below}{RGB}{102,0,204}
\definecolor{cascade}{RGB}{204,0,0}
\definecolor{iid}{RGB}{153,51,0}

\theoremstyle{remark}
\newtheorem*{remark}{Remark}

\def\paragraph#1{\noindent \textbf{#1}}

\numberwithin{equation}{section}

\def\d{\mathrm{d}}

\def\<{\langle}
\def\>{\rangle}

\def\a{\alpha}
\def\b{\beta}
\def\e{\epsilon}

\def\g{\gamma}

\def\s{\sigma}
\def\t{\tau}

\def\D{\Delta}

\def\G{\Gamma}

\def\S{\Sigma}

\def\del{\partial}
\def\R{{\Bbb R}}  
\def\N{{\Bbb N}}  
\def\P{{\Bbb P}}

\def\E{{\Bbb E}}

\let\cal=\mathcal

\def\EE{{\cal E}}
\def\FF{{\cal F}}
\def\GG{{\cal G}}
\def\HH{{\cal H}}

\def\LL{{\cal L}}

\def\TT{{\cal T}}

\def\YY{{\cal Y}}
\def\ZZ{{\cal Z}}

 \def \G {{\Gamma}}

 \def \b {{\beta}}
 \def \s {{\sigma}}

 \def \D {{\Delta}}
 \def \t {{\tau}}

 \def \g {{\gamma}}
 
 \def \d {{\delta}}
 \def \a {{\alpha}}

 \def \del {{\partial}}

 \def \ba {\begin{array}}
 \def \ea {\end{array}}

 \newcommand{\be}{\begin{equation}}
 \newcommand{\ee}{\end{equation}}

\newcommand{\bea}{\begin{eqnarray}}
 \newcommand{\eea}{\end{eqnarray}}
\def\TH(#1){\label{#1}}\def\thv(#1){\ref{#1}}
\def\Eq(#1){\label{#1}}\def\eqv(#1){(\ref{#1})}

\def\sfrac#1#2{{\textstyle{#1\over #2}}}

 \def \1{\mathbbm{1}}
\def\wt {\widetilde}

\def \zet {{\mathfrak z}}

\def\eee{{\mathrm e}}

\def\tt{t^*}
\def \log{\ln}
 \def\ti{t}
\begin{document}

 \title[From Poisson to the cascade in BBM. ]
{From $1$ to $6$: a finer analysis of perturbed branching Brownian motion}
\author[A. Bovier]{Anton Bovier}
 \address{A. Bovier\\Institut f\"ur Angewandte Mathematik\\
Rheinische Friedrich-Wilhelms-Universität\\ Endenicher Allee 60\\ 53115 Bonn, Germany }
\email{bovier@uni-bonn.de}
\author[L. Hartung]{Lisa Hartung}

 \address{L. Hartung\\ Institut für Mathematik \\Johannes Gutenberg-Universität Mainz\\
Staudingerweg 9,
55099 Mainz, Germany}
\email{lhartung@uni-mainz.de}



\date{\today}

 \begin{abstract} 
The logarithmic correction for the order of the maximum for two-speed branching Brownian motion changes discontinuously when approaching slopes $\sigma_1^2=\sigma_2^2=1$ which corresponds to standard branching Brownian motion. 
In this article we study this transition more closely by choosing $\s_1^2=1\pm t^{-\alpha}$ and $\s_2^2=1\pm t^{-\alpha}$. We show that the logarithmic correction for the order of the maximum now smoothly interpolates between the correction in the iid case $\frac{1}{2\sqrt 2}\ln(t),\;\frac{3}{2\sqrt 2}\ln(t)$ and $\frac{6}{2\sqrt 2}\ln(t)$ when $0<\a<\frac{1}{2}$. This is due to the localisation of extremal particles at the time of speed change which depends on $\a$ and differs from the one in standard branching Brownian motion.
We also establish in all cases the asymptotic law of the maximum and characterise the extremal process, which turns out to coincide essentially with that of standard branching Brownian motion. 

  \end{abstract}

\thanks{A.B. is partially supported through the German Research Foundation in 
the Collaborative Research Center 1060  ``The Mathematics of Emergent Effects" and Germany's  Excellence Strategy  -- EXC 2047 --  ``Hausdorff Center for Mathematics" at Bonn University.  This work was done during visits of A.B. at the Courant Institute, N.Y. and L.H. at the IAM at Bonn University
while L.H. was Courant Instructor. We thank both institutions for their hospitality.
}

 \maketitle


\section{Introduction}

So-called log-correlated (Gaussian) processes have received considerable attention over the last 
years,
see e.g. \cite{kistler2015,arguin2017,ABBS,BisLou16,BisLou18}. One of the reasons for this is that they represent processes where the correlations are 
on the 
borderline of becoming relevant for the properties of the extremes of the process. A paradigmatic 
example for such processes is branching Brownian motion (BBM) \cite{Moyal62,AdkeMoyal63}.
This process has been intensly investigated form the point of view of extreme value theory 
over the last 40 year, see,  e.g.,
\cite{B_M,LS,chauvin88,chauvin90,ABK_G,ABK_P,ABK_E,ABBS,CHL17,bbm-book}.
To understand what we mean by BBM being borderline, it is useful to 
consider BBM as a special case of a class of Gaussian processes labelled by a function 
 $A:[0,1]\to [0,1]$ with $A(0)=0,A(1)=1$ which is increasing and right-continuous.  
 Given such a function, so-called \emph{variable speed} branching Brownian motion \cite{DerriSpohn88,FZ_BM, MZ,BovHar13,BH14.1}
 can then be constructed in two equivalent  ways\footnote{Actually, it can be constructed in three 
 different ways: instead of making a time change in the Brownian motions, one can alternatively make 
 the branching rates explicitly time-dependent.}.

Fix a time horizon $t$ and let 
\be\Eq(intspeed.2)
\S^2_t(s)= t A(s/t),\quad s\in [0,t].
\ee
Define Brownian motion with speed function $\S^2_t$ as a time change of ordinary Brownian motion on  
$[0,t]$ as 
\be\Eq(speedy.1)
B^\S_s=B_{\S^2_t(s)}.
\ee
Branching Brownian motion with speed function $\S^2_t$ is 
constructed like ordinary branching Brownian motion except that, if a particle 
splits at some time $s<t$, then the offspring particles perform 
variable speed Brownian motion with speed function $\S^2_t$, i.e. 
their laws are independent copies 
$\{B^\S_r-B^\S_s\}_{t\geq r\geq s}$, all starting at the position of the 
parent particle at time $s$.
We assume here and throughout this paper that particles in BBM branch after an exponential time of parameter one with 
probability $p_k$ into $k$ independent copies of themselves where the branching law $p_k$ 
satisfies  $\sum_{i=1}^\infty p_k=1$,
$\sum_{k=1}^\infty k p_k=2$ and $K=\sum_{k=1}^\infty k(k-1)p_k<\infty$
This ensures, in particular that process cannot die out. It also normalises that number of particles at time 
$t$, $n(t)$ to satisfy $\E [n(t)]=\eee^t$.

Alternatively, variable speed BBM can be constructed as a Gaussian process indexed by a 
continuous time 
Galton-Watson tree with mean zero and covariances
\be\Eq(variance.2.1)
\E\left[ x_k(s)x_\ell(r)\right]= \S^2_t\left(d(x_k(t),x_\ell(t))\wedge s\wedge r\right).
\ee
where the $x_k$ label the $n(t)$ particles present at time $t$ and $d(x_k(t),x_\ell(t))$ is the 
time of the most recent common ancestor of the particles labeled $k$ and $\ell$ in the
Galton-Watson tree.

The authors interest in this model was actually sparked by the above definition. As it coincides with the generalised random energy model (GREM) introduced by Derrida \cite{GD86b} (and \cite{DerriSpohn88}  on a continuous time Galton-Watson tree. These models where introduced as toy models for spin glasses for which in particular the structure of extreme values is important. In particular, the interplay between the structure of extremes and the covariance function is a major goal. In view of this, understanding these relevant questions on a tree (where correlations are easier to handle) is a key step. A first analysis on the order of the maximum for step functions was carried out in \cite{BK1,BK2}. Already in this work  the phase transition happening at the identity function (which is described in more detail below) is visible. This is a main motivation for the study of arbitrary covariance functions and in particular  this work as it sheds light on how this transition exactly happens on a microscopic level. 

After this small detour let us now connect the two definitions of branching Brownian motion.
The case $A(x)=x$ corresponds to standard Brownian motion. 
The behaviour of the extremes of these processes are dramatically different according to whether 
$A$ stays below $x$ or whether it crosses this line. 
\begin{itemize}
\item [(i)] if $A(x)<x$ for all $x\in (0,1)$, then to first sub-leading order, 
\be\Eq(order.1)
\max_{k\leq n(t)} x_k(t) \approx \sqrt 2t -\frac 1{2\sqrt 2} \ln t,
\ee
\item [(ii)] if $A(x)=x$, then Bramson \cite{B_M,B_C} has shown that
\be\Eq(order.2)
\max_{k\leq n(t)} x_k(t) \approx \sqrt 2t -\frac 3{2\sqrt 2} \ln t,
\ee
\item [(iii)]
if for some $x\in (0,1)$, $A(x)>x$, then to leading order
\be\Eq(order.1.1)
\max_{k\leq n(t)} x_k(t) \approx \sqrt 2 t\int_0^1 \sqrt {\bar A'(y)}dy,
\ee
where $\bar A$ denotes the concave hull of the function $A$. 
The sub-leading corrections depend on the details of the function $\bar A$.
For instance, if $A$ is piecewise linear with slopes $\s_1^2$ and $\s_2^2$ 
(and necessarily $\s_1^2>\s_2^2$ to be in this sub-case)on $[0,1/2)$, resp. $[1/2,1]$, 
then the correction is given by (see e.g. \cite{FZ_BM})
\be\Eq(log.1)
-\frac 3{2\sqrt 2}  ({\s_1+\s_2}) \ln t.
\ee
\end{itemize} 
Note that, as a functional of the function $A$,  the linear term  in $t$ is continuous, but the coefficient 
multiplying $\ln t$ is discontinuous at the function $A(x)=x$. 
For instance, in the example above with two speeds,  the
limit of this coefficient is
 \be\Eq(better.1)
 \begin{cases}   
 \frac 1{2\sqrt 2},&\;\text{if}\; \s_1^2\uparrow 1,\\
  \frac 3{2\sqrt 2},&\;\text{if}\;  \s_1^2= 1,\\
  \frac 6{2\sqrt 2},&\;\text{if}\  \s_1^2\downarrow 1.
 \end{cases}
 \ee
If different sequences of functions $A$ that converge to $A(x)=x$ from above
are considered, a huge variety of limiting values can be produced. 

Branching Brownian motion has strong connections to the F-KPP equation which is a well-known reaction diffusion equation admitting travelling wave solutions, 
\be\Eq(fkpp.2)
\del_tu = {1\over 2} \del_x^2u +  F(u), 
\ee
where $F$ depends on the branching law. 
This connection can be extended to variable speed branching Brownian motion in which case one obtains  the time-inhomogeneous F-KPP equation, 
\be\Eq(lisa.102)
\del_su_t = {1\over 2}\s^2(s/t) \del_x^2u_t +  F(u_t), 
\ee
where $\s^s(s/t)=\del_s \S_t^2(s)$.
Note that \eqv(lisa.102)  is really a family of pdes indexed by $t\in \R_+$, and $u_t:[0,T]\times \R\rightarrow \R$. Eq. \eqv(lisa.102) was 
studied in \cite{NRR2015}. While in the standard F-KPP case the issue is to find a scale function $m(s)$ such that, for suitable initial conditions
$us,x+m(s))$ converges to a travelling wave, in the time inhomogeneous case where are strictly speaking no travelling waves. However, one can still analyse the
"front" position by defining $X(t)=\sup(x:u_t(t,x)=1/2)$ and show that
$u_t(t,x+X(t))$ converges to some limiting profile. By \eqv(fkpp.2), this then still gives the law of the maximum, resp. other functionals related to 
variable speed BBM.

Further properties, in particular the laws of the rescaled maxima and the extremal processes
are fully understood in the cases when $A(x)\leq x$ for all $x\in [0,1]$ and in the case when 
$\bar A$ is a piecewise linear function \cite{BovHar13,BH14.1}.. 

In this paper we have a closer look at the apparent discontinuities that happen when $A$ crosses the identity line (see \eqv(better.1)) . 
 For this, we consider functions $A=A_t$ that depend explicitly on the 
time horizon $t$. Kistler and Schmidt \cite{KistSchmi15} have considered the case then 
$A_t$ is a step function with step sizes $t^\a$ and step heights $t^\a$ that converges to $A(x)=x$ from 
below. They showed that in this case, the logarithmic correction is 
given by $\frac {3-2\a}{2\sqrt 2}\ln t$ which interpolates nicely between the cases (i) and (ii). 

Here we consider piecewise linear functions that 
lie slightly above or below $A(x)=x$.  
More precisely, we restrict ourselves to the simplest example, where
\be\Eq(function.1)
A_t(x)=\begin{cases}
     \s_1^2(t) x, &\text{if}\, x<1/2,\\
      \s_1^2(t)/2+\s_2^2(t)(x-1/2),&\text{if}\, x\geq 1/2,
     \end{cases}
     \ee
     with $\s_1^2(t)=1\pm t^{-\a}$ and $\s_2^2(t)=1\mp t^{-\a}$.
    Different cases can be treated using essentially the same techniques, if necessary in an iterative way.

     In this case, we will show that 
    \begin{itemize} 
    \item [(i)] If $\s_1^2(t)=1-t^{-\a}$, the leading term is $\sqrt 2t$ for all $\a>0$, and the 
    logarithmic corrections are
    \be\Eq(correction-below.1)
  -  \begin{cases} \frac {1+4\a}{2\sqrt 2}\ln t,&\;\; \text{if}\,\, \a\in (0,1/2],
    \\ \frac {3}{2\sqrt 2}\ln t,&\;\; \text{if}\,\, \a\in [1/2,\infty),
    \end{cases}
    \ee
    \item[(ii)]
    If $\s_1^2(t)=1+t^{-\a}$,
     the leading term is $\sqrt 2 \frac {\s_1+\s_2}2t$\footnote{Note that $
         \sqrt 2 \frac {\s_1+\s_2}2t \approx \sqrt 2 (t+t^{1-2\a})$,  which is already different
         from the BBM case if $\a\leq 1/2$.},
      and the logarithmic correction is 
    \be \Eq(correction.1)
 -   \begin{cases}
    \frac 3{2\sqrt 2}  (\s_1 +\s_2(1-2\a))\ln t \approx    \frac 3{2\sqrt 2}  (2 -2\a)\ln t ,  &\text{ if}\, \a\in [0,1/2),\\
     \frac 3{2\sqrt 2}  \s_1 \ln t\approx    \frac 3{2\sqrt 2}\ln t, &\text{ if}\, \a\geq 1/2.
     \end{cases}
     \ee
          \end{itemize}
          
Interpreting this result in context of the F-KPP equation this hints at a continuity result for the speed of the front positions. 
   
   \subsection*{Localisation.}
The key observation that will be needed to prove this and more detailed facts is
a localisation result on the position of the ancestors of extremal particles a time $t/2$. 
It is known that in the case when $\s_1^2=1+O(1)$, the ancestors of extremal particles at time $t$ 
are also extremal at time $t/2$, and so are just a logarithm of $t$ below $\sqrt 2t \s_1$. For standard 
BBM, these particles will be $O(\sqrt t)$ below $\sqrt 2 t/2$. In the case $\s_1^2=1-O(1)$, 
these particles are even  further below, namely by $\sqrt 2(\s_1-\s_1^2)t/2$ \cite{BovHar13}.
We will show (in Chapters 3 and 4, resp.), that the ancestors of extremal particles at time $t$ are 
below $\sqrt  2\s_1 t/2$ by $O(t^\a)$, in the case $\s_1^2=1+t^{-\a}$, and 
by $\sqrt{2}t^{1-\a}/4+O(\sqrt{t})$, in the case $\s_1^2=1-t^{-\a}$, when $\a\in (0,1/2]$(see Figure \ref{loc}). 
\begin{figure}\label{loc}
\begin{center}\hspace{-1cm}
\includegraphics[width=8.5cm]{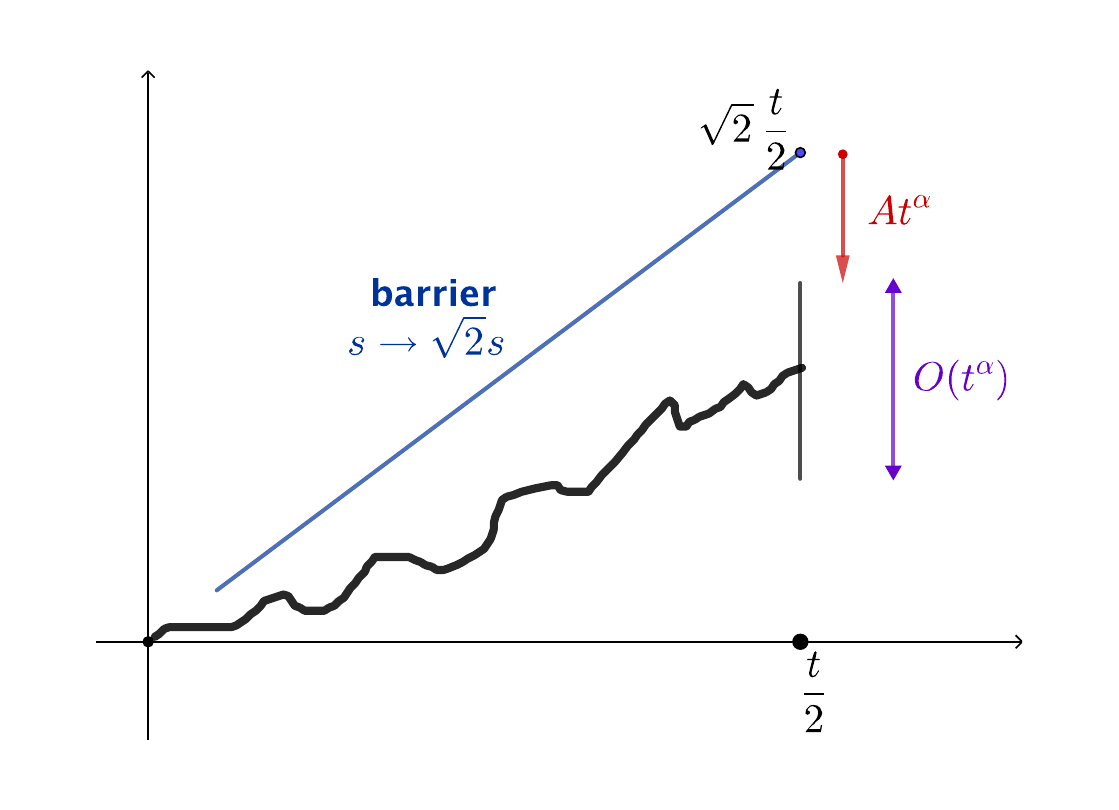}\hspace{-1.5cm}
\includegraphics[width=8.5cm]{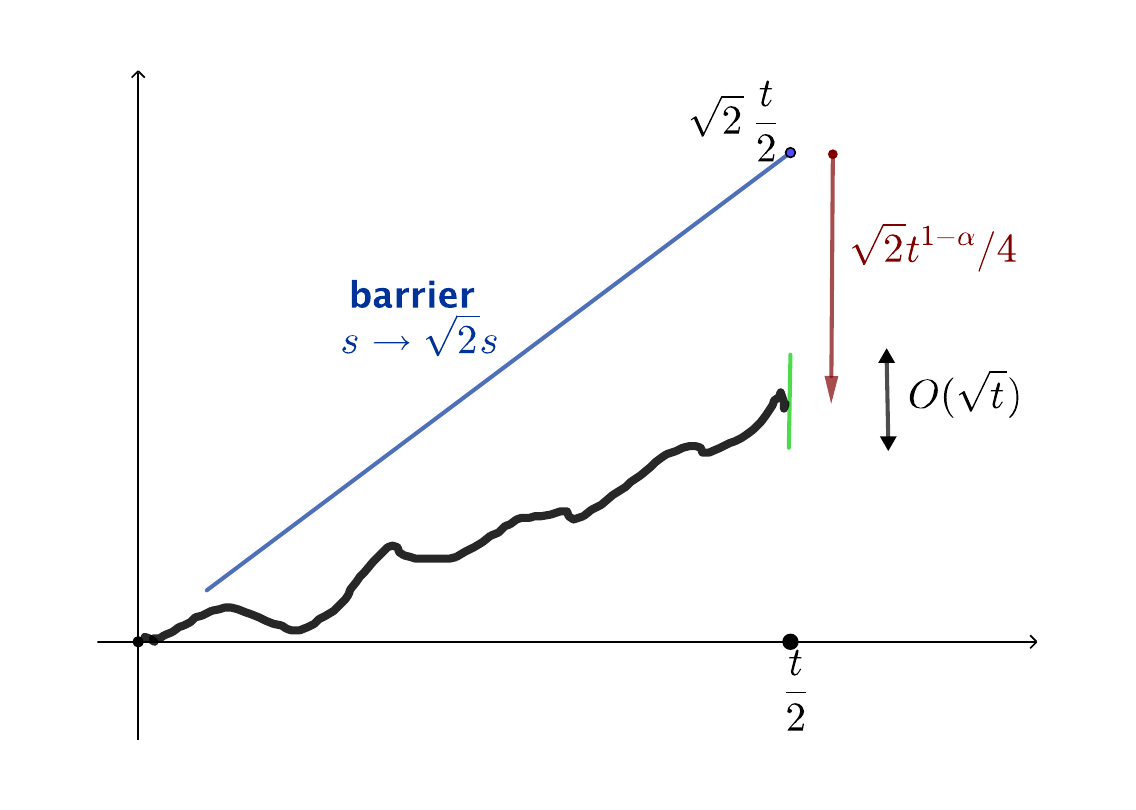}
\caption {Localisation: If the speeds are decreasing (left) then an extremal particle is $O(t^\alpha)$ below the maximum at the time of the speed change. Until this time it has to stay below the barrier $s\to\sqrt{2}s,\;s>r$. In the case of increasing speeds (right) an extremal particle is $\sqrt{2}t^{1-\a}/4\pm O(\sqrt{t})$ below the maximum at the time of the speed change. Until then it has again to stay below the barrier. }
\end{center}
\end{figure}
Afficionados of BBM will readily infer \eqv(correction-below.1) and \eqv(correction.1)
 from this information.
To actually prove this is, however, a bit more delicate.  The basic strategy is similar to that used in 
the 
case of two-speed BBM with $\s_1^2<1$ in \cite{BovHar13}, but there are some interesting twists.

Apart from the analysis of the log-correction to the value of the maximum, we also
analyse the law of the maximum and the  nature of the extremal process in 
these cases. Of course, in both cases the law of the maximum converges to 
a randomly shifted Gumbel distribution. Less obviously, whenever $\a\in (0,1/2)$, 
the random shift is always given by the \emph{derivative martingale} (see \eqv(derimar.1) below).  The extremal process has the same 
structure as in BBM, i.e. a decorated Cox-process, where the decoration process is independent of  $\a$.

In the remainder of this paper, when we consider the case $\s_1>\s_2$, we always set
$
\s_1^2=1+t^{-\a}, \s_2^2=1-t^{-\a}$, and 
\be\Eq(m.1)
m(t)=m_\a^+(t)=\sqrt 2\frac {\s_1+\s_2}2 t-\frac 3{2\sqrt 2}(2-2\a)
\ln \ti
\ee
In the case $\s_1<\s_2$, we will set 
\be\Eq(b.1)
m(t)=m_\a^-(t)=\sqrt 2 t -\frac {1+4\a}{2\sqrt 2}\ln t, 
\ee
In both cases, this is correct for $0<\a\leq 1/2$. If $\a>1/2$, all is exactly as in 
standard BBM.

We will denote particles of two speed BBM with variances 
$\s_1^2$ on $[0,t/2]$ and $\s_2^2$ on $[t/2,t]$ by 
$\tilde x_k(s)$ and those of standard BBM by $x_k(s)$. 

Before stating the main result of this paper, let us recall the two key martingales that were introduced by Lalley and Sellke \cite{LS},
the derivative martingale, $Z(t)$, and (what we like to call) the McKean martingale, $Y_\s(t)$. 
They are defined in terms of a standard BBM $x(t)$ via
\be\Eq(derimar.1)
Z(t)\equiv \sum_{k=1}^{n(t)}(\sqrt 2 t-x_k(t))\eee^{\sqrt 2(x_k(t)-\sqrt 2t)},
\ee
and 
\be
\Eq(derimar.2)
Y_{\s}(t)\equiv \sum_{k=1}^{n(t)}\eee^{\sqrt 2 \s x_k(t)-(1+\s^2)t},
\ee
Lalley and Sellke have shown that $Z(t)$ converges, as $t\uparrow 
\infty$, to an a.s.\ positive random variable $Z$, while 
for $\s\geq1$, $Y_\s(t)$ converges a.s.\ to zero. On the other hand, if
$\s<1$, then $Y_\s(t)$ is uniformly integrable and converges to a random variable $Y_{\s}$ 
(see \cite{BovHar13}).

We can now state the main results of this paper.

\begin{theorem}\TH(main.1)
Let $\tilde x(t)$ be two-speed BBM with $\s_1^2(t)=1\pm t^{-\a}$ and $\s_2^2=1\mp t^{-\a}$. Then, if $\a\in (0,1/2)$, 
for all $y\in \R$,
\be\Eq(main.1.1)
\lim_{t\uparrow\infty}\P\left(\max_{k\leq n(t)} \tilde x_k(t)-m_\a^\pm(t)\leq y\right)
= \begin{cases}
\E_Z \Bigl[\eee^{-\frac{2C Z}{\sqrt{\pi}}\eee^{-\sqrt 2y}}\Bigr], &\, \text{in the $+$ case},\\
\E_Z \Bigl[\eee^{-{C Z}\eee^{-\sqrt 2y}}\Bigr], &\, \text{in the $-$ case},
\end{cases}
\ee
where $Z$ is the limit of the  derivative martingale
(c.f. \eqv(derimar.1) and $C$ is the positive constant
\be
C= \lim_{r\uparrow\infty} \sqrt {\frac 2\pi} \int_0^\infty u(r,y+\sqrt 2r) \eee^{\sqrt 2y}dy,
\ee
where $u$ is the solution of the F-KPP equation with Heaviside initial conditions.
\end{theorem}

Similarly, we get the convergence of Laplace functionals, that then imply the convergence of the 
extremal process.

\begin{theorem}\TH(main.2)
Under the same hypotheses as in Theorem \thv(main.1), for any  bounded 
non-negative function, $\phi$, with with compact support, for all $y\in \R$,
\be\Eq(main.2.1)
\lim_{t\uparrow\infty}\E\left[\eee^{-\sum_{k=1}^{n(t)} \phi(\tilde x_k(t)-m_\a^{\pm}-y)}\right]
= \begin{cases}
\E_Z \Bigl[\eee^{-\frac{2C(\phi)Z}{\sqrt{\pi}}\eee^{-\sqrt 2y}}\Bigr], &\, \text{in the $+$ case},\\
\E_Z \Bigl[\eee^{-C (\phi)Z\eee^{-\sqrt 2y}}\Bigr], &\, \text{in the $-$ case},
\end{cases}
\ee
where  $Z$ is the limit of the derivative martingale and $C(\phi)$ is the positive constant
\be
C(\phi)= \lim_{r\uparrow\infty} \sqrt {\frac 2\pi} \int_0^\infty u(r,y+\sqrt 2r) \eee^{\sqrt 2y}dy,
\ee
where $u$ is the solution of the F-KPP equation with initial condition $u(y,0) =\exp(-\phi(-y))$.
\end{theorem}

\begin{remark} Theorem \thv(main.2) implies that the extremal process is, up to a constant shift,  always the same as that of standard BBM 
(see \cite{ABK_E}), if $\a>0$.
\end{remark}

\begin{figure}
\begin{center}
\includegraphics[width=12cm]{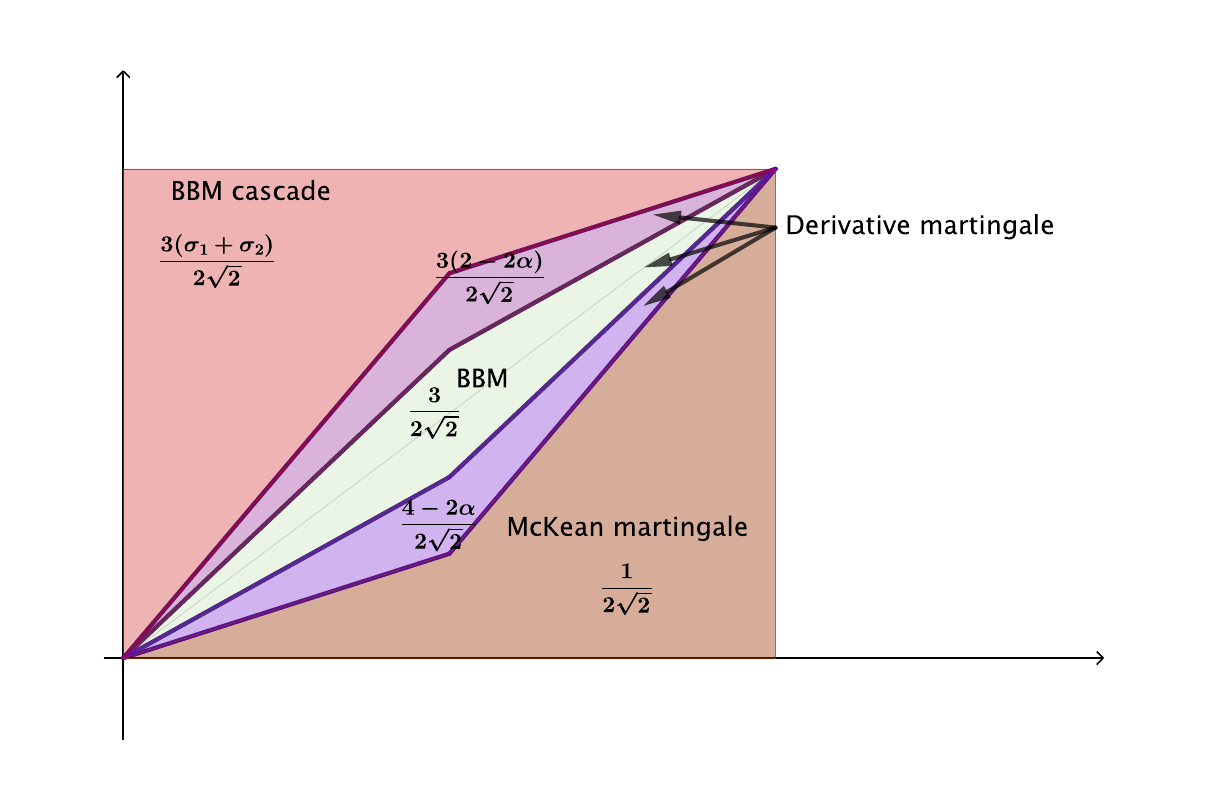}
\caption
{Phase diagram of two-speed BBM. In the {\color{bbm} inner phase ($\a>\frac{1}{2}$)}, everything is as in standard BBM. 
In the {\color{cascade} north-west regime}, the order of the maximum and the extremal process are a concatenation of two such processes
 for standard BBM. In the {\color{above} regime in between ($0<\a<\frac{1}{2}$)}, the order of the maximum interpolates smoothly between the surrounding regimes. 
In the {\color{iid} south-east regime}, the order of the maximum coincides with the one in the iid case. The extremal process is similar as the one for BBM but the martingale appearing is different. In the {\color{below} regime, with $\s_1^2=1-t^{-\a}$, $0<\a<\frac{1}{2}$ }, the order of the maximum interpolates smoothly between the iid and the BBM order of the maximum. Observe that in the three middle regimes the extremal process coincides up to constant shift with the one of standard BBM and the martingale is always the derivative martingale.}
\end{center}
\end{figure}

\subsection*{Outline of the paper.} The the remainder of this paper is organised as follows. In Section 2 we recall some facts on the tail behaviour of 
solutions of the F-KPP equation that form the crucial input in the analysis. The two following Sections 3 and 4
contain the 
proof of Theorem \thv(main.1). We deal separately with the cases $\s_1>1$ and $\s_1<1$. 
The structure of the proof is the same in both cases, but the details of the calculations are different and it appears 
easier to follow the arguments in each case rather then to jump back and forth. The way both chapters are 
organised is as follows. First, we show where the extremal particles are localised at the change-time $t/2$.
Then we exploit the branching property at time $t/2$ to set up a recursion where 
the tail asymptotics of the law of the maximum of the  BBM's after time $t/2$ are used.
This results in a formula that is already somewhat reminiscent of the Lalley-Sellke representation \cite{LS} of the
limiting distribution of the maximum of BBM. However, to prove convergence, we need to
exhibit more independence by splitting paths at time $t^\b$, for some suitable small $\b$. 
This results in an expression that in all cases involves a slight modification of the derivative martingale, 
that we then show to converge towards the limit of the  usual derivative martingale. 
In Section 5 we prove convergence of the the Laplace functionals and hence the extremal process. 
This is essentially identical to the proof of the law of the maximum and requires just a slight extension of the
results on the asymptotics of solutions of the F-KPP equation to the case of weakly $t$-dependent initial conditions.

\section {Prelimnaries about BBM}

In this section we collect some known results about standard branching Brownian motion. 
 A fundamental property of BBM is its relation to the Fisher-
Kolmogorov-Petrovsky-Piscounov  (F-KPP) equation \cite{fisher37,kpp} that was established by Ikeda, Nagasawa, and,Watanabe \cite{Ikeda1,Ikeda2,Ikeda3},  and McKean  \cite{McKean}. Namely, if we set, for some 
function $f:[0,1]\to [0,1]$,
\be\Eq(fkpp.1)
v(t,x)\equiv \E\left[\prod_{k=1}^{n(t)} f\left(x-x_k(t)\right)\right],
\ee
then $u(t,x)\equiv 1-\nu(t,x)$ is the solution of the F-KPP equation \eqv(fkpp.2)
with initial condition $u(0,x)=1-f(x)$, and
\be\Eq(fkpp.3)
F(u)=(1-u)-\sum_{k=1}^\infty p_k(1-u)^k.
\ee
The following proposition is based on the deep analysis of the behaviour of solutions to the F-KPP equation
presented in Bramson's monograph \cite{B_C}.
 
\begin{proposition}\TH(A.Prop1)
Let $u$ be a solution to the F-KPP equation with initial data satisfying 
\begin{itemize}
\item[(i)] $0\leq u(0,x)\leq 1$;
\item[(ii)] $\exists h>0$, $\limsup_{t\to\infty}\frac{1}{t}\log\int_t^{t(1+h)}u(0,y)\mbox{d}y\leq-\sqrt{2}$;
\item[(iii)] $\exists v>0,\,M>0$, and $N>0$, it holds that $\int_{x}^{x+N}u(0,y)\mbox{d}y>v$,  $\forall x\leq -M$;
\item[(iv)] moreover, $\int_0^\infty u(0,y)y\eee^{2y}\mbox{d}y<\infty$.
\end{itemize}
Then we have, for $0<x=x(t)$ such that $\lim_{t\uparrow\infty} x(t)/t=0$ 
\be\Eq(tail.1)
\lim_{t\to\infty}\eee^{\sqrt{2}x}\eee^{x^2/2t}x^{-1}u(t,x+\sqrt{2}t-\sfrac 3{2\sqrt 2}\ln t)=C,
\ee
where $C$ is a strictly positive constant that depends only on
 the initial condition $u(0,\cdot)$. 
 More precisely, 
 \be\Eq(mopre.1)
 C\equiv \lim_{r\uparrow\infty} \sqrt {\frac 2\pi} 
 \int_0^\infty u(r,y+\sqrt 2r) \eee^{\sqrt 2y}
y dy.
 \ee
\end{proposition}

\begin{proof} The proof of this proposition is a direct adaption of the proofs of the corresponding propositions in \cite{ABK_E} and \cite {BovHar13}
for the cases $x\sim \sqrt t$ and $x\sim t$.
\end{proof}

\begin{remark} Choosing for $f$ the Heaviside function, this proposition implies in particular that, for $x>0$, 
\be\Eq(probab.1)
\lim_{t\to\infty}\eee^{\sqrt{2}x}\eee^{x^2/2t}x^{-1}\P\left(\max_{k\leq n(t)}x_k(t)>
 x+\sqrt{2}t -\frac 3{2\sqrt 2}\ln t\right)=C.
\ee

\end{remark}

The following rougher bound that follows by using the many-to-one lemma and standard Gaussian asymptotics,

\begin{lemma}
\TH(probab.2)
For any $x\in \R_+$,
\be
\Eq(probab.3)
\P\left(\max_{k\leq n(t)}x_k(t)>
 x+\sqrt{2}t\right)\leq \frac {\eee^{-\sqrt 2x-\frac{x^2}{2t}}}{\sqrt {2\pi}( \sqrt {2t} +x/\sqrt t)}.
 \ee
\end{lemma}

\def\CCC{\wt C}
\section{The law of the maximum: the case $\s_1^2=1+t^{-\a}$}

The aim of this section is to prove Theorem \thv(main.1) in the case when $\s_1^2=1+t^{-\a}$, i.e.
to show 
 that
\be\Eq(M.Th2)
\lim_{t\to\infty}\P\left[\max_{1\leq k \leq n(t)}\tilde x_k(t)-m_\a^+(t)\leq y\right]= \E\left[\exp\left(-\frac {\CCC Z}{\sqrt{2\pi}}
\eee^{-\sqrt{2}y}\right)\right],
\ee
where we set $\CCC\equiv 2^{3/2} C$. In this section we will always write $m(t)\equiv m^+_\a(t)$.

\subsection{Localisation of paths}

To prove \eqv(M.Th2), we need to control the position of particles until time $t/2$. To this end, 
we  define three sets on the space of paths, $X:\R_+\rightarrow \R$. The first  controls the position at  time $s$. The second ensures that the path of the particle does not exceed a certain value, and the third controls the positions of particles at time $t^\b$.
\bea\Eq(D.11)
\GG_{s,A,B,\g}&=&\bigl\{X\big \vert X(s)-\sqrt{2}s\in [-A  s^\g,-Bs^\g]\bigr\},\nonumber\\
\TT_{s_1,s_2}&=&\bigl\{X \big \vert \forall_{s_1\leq q\leq s_2} X(q)\leq \sqrt 2q \bigr\},\nonumber\\
\HH_\d &=& \bigl\{X \big \vert  X(t^\b)\leq \sqrt 2t^\b-t^{\b\d}\bigr\}.
\eea

 In the case of standard BBM, it was shown in Bramson \cite{B_M} (see also
the detailed analysis in 
 \cite {ABK_G})
 that the positions of particles that are near the maximum at time $t$ are at time $t/2$ in
a window of order $\sqrt {t}$ below $\sqrt 2 t/2$. In the case of 2-speed BBM with 
$\s_1<\s_2$, it was shown in \cite{BovHar13} that  the corresponding window is
 of width $\sqrt t$ around $\sqrt 2\s_1^2t/2$, which is a linear order in  $t$ below the 
level of the maximal particles at time $t/2$ (which is near $\sqrt 2 \s_1 t/2$). 
If $\s_1>1$, then extremal particles descend from the actual extremal particles at time $t/2$. So we expect that 
in our case, we see a transition from $\sqrt t$ to "zero" as we vary $\a$. 

\begin{proposition}\TH(P.Prop0) Let $\s_1^2=1+t^{-\a}, \s_2^2=1-t^{-\a}$. For 
any  $\e>0$, there is $r_0<\infty$, such that for all $r>r_0$, for all $t$ large enough,
\be\Eq(p.0)
\P\left[\exists_{j\leq n(t)}:   \s_1^{-1}\tilde x_j\in \TT_{r,t/2}\right]\leq \e.
\ee
\end{proposition}
\begin{proof}
The event considered depends only on standard BBM up to time $t/2$. The well know estimate for standard BBM
 follows from 
 Bramson's results in \cite{B_M}, see also \cite{ABK_G}. 
\end{proof}

The next proposition states that extremal particles stay by $t^\a$ below $\sqrt 2 t/2$ at
 time $t/2$ when the speed change happens. 

\begin{proposition}\TH(P.Prop1) Let $\s_1^2=1+t^{-\a}, \s_2^2=1-t^{-\a}$. For any $d\in \R$ and 
any  $\e>0$,  there exists a constant $A,B>0$ such that, for all $t$ large enough,
\be\Eq(p.1)
\P\left[\exists_{j\leq n(t)}: \{\tilde x_j(t)>m(t)-d \}\land \{  \s_1^{-1}\tilde x_j\in \GG_{t/2,A,B,\a}\}\right]\leq \e.
\ee
\end{proposition}

\begin{proof} Abbreviate $I\equiv [\sqrt 2t/ 2-At^\a,\sqrt 2t/ 2-Bt^\a]$.
The probability in question can be written in the form
\be\Eq(p.2)
\P\left(\exists_{k\leq n(t/2)}  : \{\s_1x_k(t/2)>  m(t)-\s_2\max_{\ell\leq n^k(t/2)}x_\ell^k(t/2)-d\} \land 
\{x_k(t/2) \not \in  I\}\right),
\ee
where we denote by $x^k$, $k\in \N$  iid copies of standard BBMs.
We can also insert the condition $\TT_{r,t/2}$ at no cost by Proposition \thv(P.Prop0).
Then the 
expression in \eqv(p.2) becomes
\be\Eq(p.2.1)
\P\Bigl(\exists_{k\leq n(t/2)}  : \{\s_1\max_{\ell\leq n^k(t/2)}x_\ell^k(t/2) >  m(t)- \s_2x_k(t/2)-d \}\land
\{x_k(t/2)\not \in I\}
\land \{ x_k(s)\leq \sqrt 2s, \forall_{s\in [r,t/2]}\}
\Bigr).
\ee
By the many-to-one lemma, 
this is bounded from above by
\bea\Eq(p.3)
    &&\eee^{t/2}\E\left[ \1_{x_1(t/2)\not\in I}
    \1_{x_k(s)\leq \sqrt 2s, \forall s\in [r,t/2]}
    \1_{ \max_{\ell\leq n^k(t/2)} \s_2x_\ell^k(t/2) >  m(t)- \s_1x_k(t/2)-d }\right]\\\nonumber
    &&=\eee^{t/2}\int_{I^c} \frac{  \eee^{-\frac{z^2}{t}}}{\sqrt {\pi t}}
    \P\left(\zet^{t/2}_{0,\sqrt 2 t/2-z}(s)\leq 0, \forall_{s\in (r,t/2]}\right)
    \P\left( \s_2\max_{\ell\leq n^k(t/2)}x_\ell^k(t/2) >  m(t)- \s_1z-d \right)dz,
    \eea
where $\zet^{t/2}_{0, y}$ denotes the Brownian bridge from $0$ to $y$ in time $t/2$
and we wrote $I^c$ short for $ I^c\cap (-\infty,\sqrt 2t/2]$.
The probability regarding the Brownian bridge satisfies 
\be\Eq(p.4)
    \P\left(\zet^{t/2}_{0,\sqrt 2 t/2-z}(s)\leq 0, \forall_{s\in (r,t/2]}\right)
    \leq   \sqrt{\frac 2\pi} \frac {(\sqrt 2 t/2-z)\sqrt r} {t/2},
    \ee
    as long as $\sqrt 2 t/2-z\ll \sqrt t$, and is bounded by $1$ otherwise.
    We now split the integral into the parts where $z$ is above $\sqrt 2 t/2-Bt^\a$ and  where it
    is below $\sqrt 2 t/2-At^\a$.
    The first part gives, with a change of variables, 
 \bea\Eq(p.5)
 \nonumber
          &&
    \eee^{t/2}\int_0^{Bt^\a} \frac{  \eee^{-\frac{(y-\sqrt2 t/2)^2}{t}}}{\sqrt {\pi t}}
   \P\left(\zet^{t/2}_{0,-y}(s)\leq 0, \forall_{s\in (r,t/2]}\right)
   \\
    &&\quad\quad\times\;
    \P\left(  \s_2\max_{\ell\leq n^k(t/2)}x_\ell^k(t/2) >  m(t)-\s_1\sqrt 2t/2+ \s_1 y-d \right)dy.
    \eea
    Recall that 
    $m(t)- \s_1\sqrt 2t/2 = \s_2\sqrt 2 t/2 -\frac {3}{2\sqrt 2} (\s_1+\s_2(1-2\a))\ln t$. 
    Hence, the probability involving the maximum in \eqv(p.5) reads
    \be\Eq(p.6)
    \P\left( \max_{\ell\leq n^k(t/2)}x_\ell^k(t/2) >  \sqrt 2t/2-\frac 3{2\sqrt 2}\left( \frac{\s_1}{\s_2} 
    + (1-2\a) \right)\ln t +\frac{\s_1}{\s_2}y-d/\s_2 \right).
    \ee
    Using Proposition  \thv(A.Prop1), respectively \eqv(probab.1),
    we see that this probability equals, asymptotically as $t\uparrow\infty$, for $y> \frac 3{2\sqrt 2}(1-2\a) \ln t+d/\s_2$,
    \bea\Eq(p.6)
   && C\left(\sfrac {\s_1}{\s_2}y-\sfrac 3{2\sqrt 2} (1-2\a) \ln t -d/\s_2\right)
    \eee^{-\sqrt 2\left(-\frac 3{2\sqrt 2}  (1-2\a) \ln t +\sfrac {\s_1}{\s_2}y-d/\s_2\right)}\nonumber\\
    &&=t^{3/2} C\left(\sfrac {\s_1}{\s_2}y-\sfrac 3{2\sqrt 2} (1-2\a) \ln t -d/\s_2\right)
    \eee^{-\sqrt 2\left(\sfrac {3\a}{\sqrt 2}
     \ln t +\sfrac {\s_1}{\s_2}y-d/\s_2\right)}\nonumber
    \\ &&=
    C\left(\sfrac {\s_1}{\s_2}y-\sfrac 3{2\sqrt 2}\left((1-2\a) \right)\ln t-d/\s_2\right)
    t^{3/2-3\a } \eee^{-\sqrt 2\left(\sfrac {\s_1}{\s_2}y-d/\s_2\right)}.
       \eea
       For $y\leq \frac 3{2\sqrt 2}(1-2\a) \ln t+d/\s_2$, we simply bound the probability by $1$.
       Inserting this and the bound \eqv(p.4) into \eqv(p.5), we see that this term is not larger than 
       \bea\Eq(p.7)
     &&\int_{\sfrac 3{2\sqrt 2}(1-2\a) \ln t+d/\s_2}^{Bt^\a}
  \frac{  \eee^{y\sqrt 2  }}{\sqrt {\pi t}}
   \frac {y\sqrt r} {\sqrt {2\pi}t}
     C\left(\sfrac {\s_1}{\s_2}y-\sfrac 3{2\sqrt 2} (1-2\a) \ln t -d/\s_2\right)
    t^{3/2 -3\a}  \eee^{-\sqrt 2\left( \sfrac{ \s_1}{\s_2}y-d/\s_2\right)}
 dy\nonumber\\
 &&
 =\int_{t^{-\a} \left(\sfrac 3{2\sqrt 2}(1-2\a) \ln t+d/\s_2 \right)}^{B} 
  \frac{  \eee^{-y \sqrt2  }}{\sqrt {\pi }}
   \frac {yt^{\a}\sqrt r} {\sqrt {2\pi}}
         C\left(\sfrac {\s_1}{\s_2}yt^\a-\sfrac 3{2\sqrt 2} (1-2\a) \ln t -d/\s_2\right)
    t^{-2\a} \eee^{\sqrt 2d/\s_2} 
 dy\nonumber\\
&&\sim \int_0^{B} 
  \frac{  \eee^{-y \sqrt2  }}{\sqrt {\pi }}
   \frac {\sqrt r} {\sqrt {2\pi}}
    C y^2  \eee^{\sqrt 2d} 
dy,
    \eea
which is finite and tends to zero, as $B\downarrow 0$. Finally, for the remaining part of the integral in \eqv(p.5) we bound by
   \be\Eq(p.7.1)
     \int_0^{\sfrac 3{2\sqrt 2}(1-2\a) \ln t+d/\s_2}
  \frac{  \eee^{y\sqrt 2  }}{\sqrt {\pi t}}
   \frac {y\sqrt r} {\sqrt {2\pi}t} dy
   \leq C(\ln t)^2 r t^{-3\a},
\ee
which  tends to zero, as $t\uparrow\infty$.
 The  part of the   integral  in \eqv(p.3)  involving the terms below $\sqrt 2t/2-At^\a$ can be written as
 \bea\Eq(p.8)
          &&\eee^{t/2}\int_{At^\a}^\infty \frac{  \eee^{-\frac{(y-\sqrt2 t/2)^2}{t}}}{\sqrt {\pi t}}
   \P\left(\zet^{t/2}_{0,-y}(s)\leq 0, \forall_{s\in (r,t/2]}\right)\nonumber\\
&&\quad\quad\times    \P\left(  \s_2\max_{\ell\leq n^k(t/2)}x_\ell^k(t/2) >  m(t)-\s_1\sqrt 2t/2+ \s_1 y-d \right)dy.
    \eea
We have to distinguish the cases where $z\leq K \sqrt t$ and the rest. In the former, we can proceed as in 
the case above and we get, up to vanishing terms, for any $K>0$, 
\be\Eq(p.9)
 \int_A^{K t^{1/2-\a} }
  \frac{  \eee^{-y \sqrt2  }}{\sqrt {\pi }}
   \frac {\sqrt r} {\sqrt {2\pi}}
    C y^2  \eee^{\sqrt 2d/\s_2}dy \rightarrow  \int_A^\infty
  \frac{  \eee^{-y \sqrt2  }}{\sqrt {\pi }}
   \frac {\sqrt r} {\sqrt {2\pi}}
    C y^2 \eee^{\sqrt 2d/\s_2}
dy,  \;\;      \text{as} \,t\uparrow\infty,
\ee
which in turn  converges to zero as $A\uparrow \infty$.
For the remaining term, it is enough to bound the probability involving the Brownian bridge by one and to use the bound \eqv(probab.3).
One then gets a bound 
\be\Eq(p.10)
  \int_{L\sqrt t}^\infty\eee^{-y\sqrt2t^{-\a} } t^{1/2+3(1-2\a)/2}dy
\leq  \int_K^\infty \eee^{-\sqrt2 y t^{1/2-\a}} t^{5/2-3\a}dy,
\ee
which tends to zero rapidly, as $t\uparrow \infty$. 
This concludes the proof.
\end{proof}

The next proposition states that  $\HH_\d$ holds for all extremal particles,
for $0<\d<1/2$. This is a weaker form of the localisation results shown in \cite{ABK_G}.

\begin{proposition}\TH(P.Prop2) Let $\s_1^2=1+t^{-\a}, \s_2^2=1-t^{-\a}$. For any $d\in \R$ and 
any  $\e>0$,  there exists $ 0<\d<1/2$  such that, for all $t$ large enough,
\be\Eq(pruz.3)
\P\left[\exists_{j\leq n(t)}: \{\tilde x_j(t)>m(t)-d\} \land  \{ \s_1^{-1}\tilde x_j\in \HH_{\d}\}\right]\leq \e.
\ee
\end{proposition}

\begin{proof} To prove this proposition, we may use Proposition \thv(P.Prop1) and the fact  that
 any path starting at zero, ending at some $ \sqrt 2t/2-z$ with $z\in [At^\a,Bt^\a]$, and 
staying below the line $\sqrt 2 s$,
will not be much above $\sqrt 2t^\b-t^{\b\d}$ at time $t^\b$.

To do so, we decompose a bridge in time $t/2$ from $0$ to $z$ into two pieces, one from 
$0$ to $\sqrt 2t^\b-y$ in time $t^\b$ and one from $\sqrt 2t^\b-y$ to $\sqrt 2t/2-z$ in 
time $\tt\equiv t/2-t^\b$. Then the probability that the first bridge stays below $\sqrt 2 s$ 
is, to leading order in $t$, given by
\be
\Eq(onebridge.1)
\sqrt {\frac 2\pi} \frac {\sqrt r y}{t^\b-r},
\ee
while the probability for the second bridge is
$2 {yz}/{\tt}$.
These estimates follow  from Lemma 2.2. in \cite{B_C}.
Thus the probability that the bridge is above $\sqrt 2t^\b-t^{\b\d}$ is 
given by 
\be\Eq(onebridge.3)
\frac {\int_{0}^{t^{\b\d}} \frac  {\eee^{-\frac {y^2}{2t^\b}}}{\sqrt {2\pi t^\b}}
\sqrt {\frac 2\pi} \frac {\sqrt r y}{t^\b-r}\frac {2yz}{\tt}dy}
 {\int_{0}^{\infty} \frac  {\eee^{-\frac {y^2}{2t^\b}}}{\sqrt {2\pi t^\b}}\sqrt {\frac 2\pi}
 \frac {\sqrt r y}{t^\b-r}\frac {2yz}{\tt}dy}
=\frac {\int_{0}^{t^{\b(\d-1/2})}   {\eee^{-\frac {y^2}{2}}}
 y^2dy}
 {\int_{0}^{\infty}   {\eee^{-\frac {y^2}{2}}}
  y^2dy}\leq  t^{3\b(\d-1/2)}.
\ee
The right-hand side tends to zero for any $\d<1/2$, which implies the assertion of the proposition.
\end{proof}

The following simple lemma shows that if a condition holds for all paths that exceed some level, 
then this condition can also be imposed on the paths when computing the probability 
that the maximum stays below that level.

\begin{lemma}\TH(minimax.1)
Let $x_k, k=1,\dots, n$ be path-valued  random variables and  $\GG$ be any event such that, for some $\e>0$,
\be\Eq(minimax.2)
\P\left(\exists_{k\leq n}:  \{x_k(t) >y\}\land \{x_k\in \GG\}\right)\geq \P\left(\exists_{k\leq n}  x_k(t) >y\right) -\e.
\ee
Then 
\be\Eq(minimax.3)
\left|\P\left(\max_{k\leq n} x_k(t) \leq y\right)-
\P\left(\max_{k\leq n:x_k\in\GG}x_k(t) \leq y\right)\right|\leq \e.
\ee
\end{lemma}

\begin{proof} 
Obviously,
\bea\Eq(minimax.4)
\P\left(\max_{k\leq n}x_k (t)\leq y\right)&\leq&
\P\left(\max_{k\leq n:x_k\in\GG}x_k(t) \leq y\right)\nonumber\\
&=&1-\P\left(\exists_{k\leq n}: \{ x_k(t) >y\}\land \{x_k\in \GG\}\right)
\nonumber\\
&\leq& 1-\P\left(\exists_{k\leq n}:  x_k(t) >y\right)+\e
=\P\left(\max_{k\leq n}x_k(t) \leq y\right)+\e,
\eea
which proves the lemma.
\end{proof}

\subsection{Recursive structure}
We want to control
\be\Eq(prep.1)
 \P\left[\max_{1\leq k \leq n(t)}\tilde x_k(t)-m(t)> y\right]= \P\left[\max_{ k \leq n(t/2), \ell\leq n^k(t/2)}  \s_1 x_k(t/2)+\s_2 x^k_\ell(t/2)-m(t)> y\right],
\ee
where, for each $k$, $(x^k_\ell(\cdot))_{l\leq n^k(t/2)}$ are particles of an independent standard branching Brownian motion.
First, we  introduce several localisation conditions in \eqv(prep.1).  
For this we need to  define  shifted versions of the event $\GG$ and $\TT$ as
\bea\Eq(D.11.2)
\GG_{s,A,B,S,T,\g}&=&\bigl\{X\big \vert X(s)-\sqrt{2}s+S\in [-A  (s+T)^\g,-B(s+T)^\g]\bigr\},\nonumber\\
\TT_{s_1,s_2,S}&=&\bigl\{X \big \vert \forall s_1\leq s\leq s_2:\; X(s) +S\leq \sqrt{2}s  \bigr\}.
\eea
By Proposition \thv(P.Prop1) we have that
\bea\Eq(prep.3) 
&&\P\left[\max_{ k \leq n(t/2), \ell\leq n^k(t/2)}  \s_1 x_k(t/2)+\s_2 x^k_\ell(t/2)-m(t)> y\right]\\\nonumber
&&\leq
\P\left[\exists_{ k \leq n(t/2),\ell\leq n^k(t/2)}: \{\s_1 x_k(t/2)+\s_2 x^k_\ell(t/2)-m(t)> y\}
\land\{ x_k\in \GG_{t/2,A,B,\alpha}\}\right]+\epsilon,
\eea
for $B$ sufficiently close to zero and $A$ large enough. The probability on the right hand side of \eqv(prep.3) 
is also  a lower bound for \eqv(prep.1). Proceeding similarly, the probability in \eqv(prep.3) is equal (up to error terms of size $\epsilon$) to
\be\Eq(prep.4)
\P\left[\exists_{ k \leq n(t/2), \ell\leq n^k(t/2)}: \{\s_1 x_k(t/2)+\s_2 x^k_\ell(t/2)-m(t)> y\}\land
\{x_k\in \GG_{t/2,A,B,\alpha}\cap\HH_{\delta}\}\right],
\ee
for any $\delta>1/2$. Moreover, we can also introduce a condition on the path between time $t^\b$
and time $t/2$  and get that \eqv(prep.4) is equal to (again up to an error of order $\epsilon$)
\be\Eq(prep.5)
\P\left[\exists_{ k \leq n(t/2), \ell\leq n^k(t/2)}: \{\s_1 x_k(t/2)+\s_2 x^k_\ell(t/2)-m(t)> y\}
\land \{x_k\in \GG_{t/2,A,B,\alpha}\cap\HH_{\delta}\cap \TT_{t^{\b},t/2}\}\right].
\ee
Set 
\be\Eq(prep.6)
\LL_{t^\b,t/2,A,B}=\GG_{t/2,A,B,\alpha}\cap\HH_{\delta}\cap \TT_{t^{\b},t/2}.
\ee
In view of Lemma \thv(minimax.1), we only need to analyse 
\be\Eq(meta.5)
\P\left[\max_{ k \leq n(t/2):x_k\in \LL_{t^\b,t/2,A,B}    , \ell\leq n^k(t/2)}  \s_1 x_k(t/2)+\s_2 x^k_\ell(t/2)-m(t)\leq y \right]
\ee
in order to prove Theorem \thv(main.1).
Using the branching property, we can rewrite \eqv(meta.5) as 
\bea\Eq(prep.2)
&&\E\Biggl[ \prod_{\stackrel{k\leq n(t/2)}{x_k\in \LL_{t^\b,t/2,A,B}}}\P\left[\max_{l\leq n^k(t/2)}x^k_l(t/2)\leq \frac{m(t)+y-\s_1 x_k(t/2)}{\s_2}\big\vert \FF_{t/2}\right]\Biggl]\\\nonumber
&&=\E\Biggl[\prod_{\stackrel{k\leq n(t/2)}{x_k\in \LL_{t^\b,t/2,A,B}}}
\left(1-\P\left[\max_{l\leq n^k(t/2)}x^k_l(t/2)> \frac{m(t)+y-\s_1 x_k(t/2)}{\s_2}\big\vert \FF_{t/2}\right]\right)\Biggl],
\eea
where $\mathcal{F}_s$ with $s\leq t/2$  denotes the $\s$-algebra generated by $\left(x(u)\right)_{u\leq s }$.

As $x_k\in \GG_{t/2,A,B,\alpha}$, we can use the tail asymptotics given in Proposition \thv(A.Prop1) to control the conditional probability in \eqv(prep.2). Namely,
\be\Eq(prep.7)
 \P\left[\max_{\ell\leq n^k(t/2)}x^k_\ell(t/2)> \frac{m(t)+y-\s_1 x_k(t/2)}{\s_2}\big\vert \FF_{t/2}\right]
 = \CCC \Gamma_k(t)\eee^{-\sqrt{2}\Gamma_k(t)}  (1+o(1)),
\ee
where the $o(1)$ error term is uniform in the range of possible values for $x_k(t/2)$ as $x_k\in \GG_{t/2,A,B,\alpha}$ and 
\bea\Eq(prep.8)
\Gamma_k(t)&=&\frac{m(t)+y-\s_1 x_k(t/2)}{\s_2}-\left(\sqrt{2}\frac{t}{2}-\frac{3}{2\sqrt{2}}\log(t)\right)\nonumber\\
&=&\sfrac{\s_1}{\s_2} \left(\sqrt{2} \frac{t}{2}-x_k(t/2)\right)-  \frac{3}{2\sqrt2}\left(1-2\alpha\right)\log(t)+y/\s_2.
\eea
Plugging \eqv(prep.7) back into \eqv(prep.2) we obtain that the expectation in \eqv(prep.2) is equal to
\bea\Eq(prep.9)
&&\E\Biggl[\prod_{\stackrel{k\leq n(t/2)}{x_k\in \LL_{t^\b,t/2,A,B}}}\left(1- \CCC \Gamma_k(t)\eee^{-\sqrt{2}\Gamma_k(t)}\right)\Biggr](1+o(1))\nonumber\\
&&= \E\Biggl[\prod_{\stackrel{k\leq n(t/2)}{x_k\in \LL_{t^\b,t/2,A,B}}} \exp\left( -\CCC \Gamma_k(t)\eee^{-\sqrt{2}\Gamma_k(t)}\right)\Biggr](1+o(1)),
\eea
 since $\Gamma_k(t)>At^{\alpha}$  as $x_k\in \GG_{t/2,A,B,\alpha} $. Next, we rewrite the expectation in  \eqv(prep.9) by conditioning on $\FF_{t^\beta}$ as
 \be\Eq(prep.10)
 \E\Biggl[\prod_{\stackrel{k\leq n(t^\beta)}{x_k\in \HH_{\delta}}} 
 \E\Biggl[\prod_{\stackrel{j\leq n^k(t/2-t^\beta)}{x_j^k\in \GG_{t/2-t^\b,A,B,x_k(t^\beta)-\sqrt{2}t^\beta,t^\beta,\alpha}\cap\TT_{0,t/2-t^\beta,x_k(t^\beta)-\sqrt{2}t^\beta}}} 
 \exp\left(- \CCC \D_k(t)\eee^{-\sqrt{2}\D_k(t)}\right)\bigg \vert \FF_{t^\beta}\Biggr]\Biggr],
 \ee
with $\GG_{t/2-t^\b,A,B,x_k(t^\beta)-\sqrt{2}t^\beta,t^\beta,\alpha}$ and $\TT_{ 0,t/2-t^\beta,x_k(t^\beta)-\sqrt{2}t^\beta}$ as defined in \eqv(D.11.2), and
\be\Eq(prep.11)
\D_k(t)=\sfrac{\s_1}{\s_2} \left(\sqrt{2} \frac{t}{2}-x_k(t^\beta)-x^k_j(t/2-t^\beta)\right)-  \frac{3}{2\sqrt2}\left(1-2\alpha\right)\log(t)+y/\s_2,
\ee
where, for each $k$, $(x^k_j(\cdot))_{l\leq n^k(t/2-t^\b)}$ are particles of an independent standard branching Brownian motion.  We set  
\be\Eq(prep.15)
\wt\LL_{t^\beta,t/2-t^\b,x_k(t^\beta)}\equiv\GG_{t/2-t^\b,A,B,x_k(t^\beta)-\sqrt{2}t^\beta,t^\beta,\alpha}\cap\TT_{0,t/2-t^\beta,x_k(t^\beta)-\sqrt{2}t^\beta}.
\ee
We rewrite the inner expectation in \eqv(prep.10) as
\bea\Eq(prep.12)
&&\E\Biggl[ \exp\Biggl(-\sum_{\stackrel{j\leq n^k(t/2-t^\beta)}{x_j^k\in \wt\LL_{t^\beta,t/2-t^\b,x_k(t^\beta)}}} 
2^{3/2} C \D_k(t)\eee^{-\sqrt{2}\D_k(t)}\Biggr)\big \vert \FF_{t^\beta}\Biggr]\\
 \nonumber
 &&=
 \E\left[ \exp\left(-\sfrac{\s_1}{\s_2}\CCC (\sqrt{2}t^\beta -x_k(t^\beta))\eee^{-\sqrt{2}\sfrac{\s_1}{\s_2}(\sqrt{2}
 t^\beta -x_k(t^\beta))} \YY_k(t)
 -\sfrac{\s_1}{\s_2}\CCC  \eee^{-\sqrt{2}\sfrac{\s_1}{\s_2}(\sqrt{2}t^\beta -
x_k(t^\beta))} \ZZ_k(t)\right)
  \big \vert \FF_{t^\beta}\right],
\eea
where
\be\Eq(prep.13)
\YY_k(t)=\sum_{\stackrel{j\leq n^k(t/2-t^\beta)}{x_j^k\in \wt\LL_{t^\beta,t/2-t^\b,x_k(t^\beta)}}} 
\eee^{-\sqrt{2}\left(\sfrac{\s_1}{\s_2} \left(\sqrt{2} \tt-x^k_j(\tt)\right)-  \frac{3}{2\sqrt2}\left(1-2\alpha\right)\log(t)+y/\s_2\right)}
\ee
and
\bea\Eq(prep.14)
&&\ZZ_k(t)=\sum_{\stackrel{j\leq n^k(t/2-t^\beta)}{x_j^k\in \wt\LL_{t^\beta,t/2-t^\b,x_k(t^\beta)}}} 
 \left(\sfrac{\s_1}{\s_2} \left(\sqrt{2} \tt-x^k_j(\tt)\right)-  \sfrac{3}{2\sqrt2}\left(1-2\alpha\right)\log(t)+y/\s_2\right) \nonumber\\
&&\qquad\qquad\times \eee^{ -\sqrt{2}\left(\sfrac{\s_1}{\s_2} \left(\sqrt{2} \tt-x^k_j(\tt)\right)-  \frac{3}{2\sqrt2}\left(1-2\alpha\right)\log(t)+y/\s_2\right)}.
\eea
Next, we want upper and lower bounds on the expression in \eqv(prep.12)
To this end we use the basic inequality
 \be\Eq(inequality)
1-x\leq \eee^{-x}\leq 1-x+\frac{1}{2}x^2,\quad x>0,
\ee
 for 
 \be\Eq(prep.16)
 x=\sfrac{\s_1}{\s_2}\CCC (\sqrt{2}t^\beta -x_k(t^\beta))\eee^{-\sqrt{2}\sfrac{\s_1}{\s_2}(\sqrt{2}t^\beta -x_k(t^\beta))} \YY_k(t)
 +\sfrac{\s_1}{\s_2}\CCC  \eee^{-\sqrt{2}\sfrac{\s_1}{\s_2}(\sqrt{2}t^\beta -x_k(t^\beta))} \ZZ_k(t).
 \ee
 As the term $\eee^{-x}$ appears in the conditional expectation with respect to $\FF_{t^\b}$, we 
 need to compute $ \E(\YY_k(t)\vert \FF_{t^\b}), \E(\ZZ_k(t)\vert \FF_{t^\b})$ and control 
 $ \E(\YY_k(t)^2\vert \FF_{t^\b})$, $ \E(\ZZ_k(t)^2\vert \FF_{t^\b})$, and  $\E(\YY_k(t)\ZZ_k(t)\vert \FF_{t^\b})$.  
 
 \subsection{Computation of the main term} 
We begin with the computation of the averages of the McKean resp. derivative 
martingale terms. 
\begin{lemma}
\TH(expectations.1)
With the notation from the last subsection, 
\be\Eq(expectations.2)
 \E\left[\YY_k(t)|\FF_{t^\b}\right] = 2^{3/2} (\sqrt {2}t^\b-x_k(t^\b)) t^{-\a} \frac {\eee^{-\sqrt 2y}}{\sqrt {2\pi}}
 \left(1+o(1)\right),
 \ee
and 
\be
 \Eq(expectations.3)
 \E\left[\ZZ_k(t)|\FF_{t^\b}\right]=  ( {\sqrt {2}t^\b-x_k(t^\b)})  \frac{\eee^{-\sqrt 2y}}{\sqrt {2\pi }}(1 +o(1)),
 \ee
 where $o(1)$ tends to zero as first $t\uparrow \infty$ and then $B\downarrow0$ and $A\uparrow \infty$.
 \end{lemma}
 
 \begin{proof}
 We start with the conditional  expectation of $\YY_k(t)$. 
  Using the many-to-one lemma, we get
 \bea\Eq(mckean.0)
 && \E\left[\YY_k(t)|\FF_{t^\b}\right]=
  \eee^{\tt}\E\Bigl[\eee^{-\sqrt{2}\left(\sfrac{\s_1}{\s_2} \left(\sqrt{2} \tt-x(\tt)\right)-  \frac{3}{2\sqrt2}\left(1-2\alpha\right)\log(t)+y/\s_2\right)}
\1_{\forall_{0\leq s\leq \tt} x(s)+x_k(t^\b)-\sqrt 2 t^\b\leq \sqrt 2 s}\nonumber\\
 && \hspace {4cm}\times  \1_{x(\tt)-\sqrt 2\tt+x_k(t^\b)-\sqrt 2 t^\b\in [-A(t/2)^\a,-B(t/2)^\a]}\Bigr].
  \eea
  The two conditions in the indicator functions can be expressed in terms of 
 a Brownian  bridge from $x_k(t^\b)$ to its endpoint $x(\tt)$ that must  stay below $\sqrt 2 s$ all the time.
  This condition produces a factor $2\frac {(\sqrt {2}t^\b-x_k(t^\b))(\sqrt2t/2-x_k(\tt))}\tt$. 
  Using the independence of the bridge  from its endpoint, this allows us write
 \bea
 \Eq(mckean.1)
 &&
 \E\left[\YY_k(t)|\FF_{t^\b}\right]\\
 &&= 2 t^{3/2-3\a}\eee^{\tt}  \int_{\sqrt 2\tt-At^\a}^{\sqrt 2\tt-Bt^\a}
 \frac {(\sqrt {2}t^\b-x_k(t^\b))(\sqrt2t/2-z)}\tt\frac{\eee^{-\frac {z^2}{2\tt}}}{\sqrt{2\pi\tt}}
 \eee^{\sqrt 2 \frac{\s_1}{\s_2}(z-\sqrt 2\tt)}dz \eee^{-\sqrt 2y/\s_2}\nonumber\\
 &&=2\eee^{\tt}\frac {\sqrt {2}t^\b-x_k(t^\b)}\tt  t^{3/2-3\a}
  \int_{-At^\a}^{-Bt^\a}
  (-z)\frac{\eee^{-\frac {(z+\tt)^2}{2\tt}}}{\sqrt{2\pi\tt}}
 \eee^{\sqrt 2 \frac{\s_1}{\s_2}z}dz\eee^{-\sqrt 2y}(1+o(1))\nonumber\\\nonumber
 &&=2^{5/2}( {\sqrt {2}t^\b-x_k(t^\b)})   t^{-3\a}
  \int_{-At^\a}^{-Bt^\a}
  (-z)\frac{\eee^{ \sqrt 2 (\frac{\s_1}{\s_2}-1) z-   \frac{z^2}{2\tt}}}{\sqrt{2\pi}}
 \eee^{\sqrt 2 \frac{\s_1}{\s_2}z}dz\eee^{-\sqrt 2y}(1+o(1)).
 \eea
 The second inequality uses that we have chosen $\b$ so small such that $t^\b\ll z$ in the domain of integration
 so that we can replace $t/2$ be $\tt$ without making a significant error. 
 In the range of integration, the term $\frac{z^2}{2\tt}$ vanishes, as $t\uparrow \infty$. 
The integral in the last line thus becomes
\be\Eq(mckean.2)
   \int_{-At^\a}^{-Bt^\a}
  (-z)\frac{\eee^{ \sqrt2 zt^{-\a}}}{\sqrt{2\pi}}dz =\frac{t^{2\a}}{\sqrt {2\pi}}
   \int_{B}^{A}
  z\eee^{- \sqrt2 z}dz.
  \ee
  As $A\uparrow\infty$ and $B\downarrow 0$, the last integral converges to $1/2$. 
  This yields \eqv(expectations.2).
 
 Next, we treat the conditional  expectation of $\ZZ_k(t)$.  It is evident from the previous calculations, that 
 the terms in front of the exponential with the logarithm and the $y$  in \eqv(prep.14) 
 will tend to zero. What is left of the conditional expectation of $\ZZ$ is 
\bea\Eq(derivative.1) 
&& 2^{5/2} ({\sqrt {2}t^\b-x_k(t^\b)})   t^{-3\a}
  \int_{-At^\a}^{-Bt^\a}
  z^2\frac{\eee^{-\sqrt 2 t^{-\a} z-   \frac{z^2}{2\tt}}}{\sqrt{2\pi}}
 dz\eee^{-\sqrt 2y}  (1+o(1)) \nonumber\\
&&=  2^{5/2} ( {\sqrt {2}t^\b-x_k(t^\b)})    \frac{\eee^{-\sqrt 2 y}}{\sqrt {2\pi}} \int_{B}^{A}
  z^2{\eee^{-\sqrt2 z}}dz(1+o(1)).
 \eea
 The last integral converges to $2^{-5/2}$, as $A\uparrow\infty, B\downarrow 0$. Thus we get 
 \eqv(expectations.3). This concludes the proof of the lemma.
 \end{proof}
 
 \begin{remark} It is curious to see that the terms $ \sqrt {2}t^\b-x_k(t^\b)$
 appear and recreate the derivative martingale as a factor of $\E \ZZ_k(t)$. 
 If we had been a bit more sloppy and used as the probability
 for the bridge just $\frac {t^{\b/2} t^{\a}}\tt$, we would instead have gotten just 
 a factor $t^{1/2}$ multiplying the McKeane martingale. But, nothing would have changed, since by a result of A\"\i d\'ekon and Shi \cite{AideShi}, this would converge in probability to 
 a limit that has the same law as the limit of the derivative martingale. 
 \end{remark}

\subsection{Controlling the second moment}
 We now  show that the expectations of the quadratic terms are bounded
by a polynomial term in $t^\a$, see \eqv(prep.18) below. 
For this it is enough to show that  $\E \left[\YY(\tt)^2\right]\leq P(t^\a)$, for $P$ some polynomial. 
Dropping all irrelevant terms that are controlled by some power of $t$, we are left with 
computing 
\be\Eq(square.1)
\E\left(\sum_{j=1}^ {n(\tt)}
\1_{x_j\in \wt\LL_{t^\beta,t/2-t^\b,x_k(t^\beta)}} \eee^{\sqrt 2 \frac{\s_1}{\s_2}\left(x_j(\tt)-\sqrt 2\tt\right)}
\right)^2.
\ee
Using the many-to-two lemma, this is bounded by
\bea
\Eq(square.2)
&&\int_0^{\tt} ds\; \eee^{\tt+s}
\int_{-\infty}^{\sqrt 2(\tt-s)} 
\frac{dw\eee^{-\frac{w^2}{2(\tt-s)}}}{\sqrt{2\pi(\tt-s)}}
\left(
\int_{\sqrt 2\tt-w-At^\a}^{\sqrt 2\tt-w-Bt^\a}
\frac{dz\eee^{-\frac{z^2}{2s}}}{\sqrt{2\pi s}}\eee^{\sqrt 2\frac{\s_1}{\s_2}(w+z-\sqrt 2\tt)}\right)^2.\nonumber\\
\eea
Shifting the $w$-integral, this equals
\bea
\Eq(square.3)
&&\int_0^{\tt} ds\; \eee^{\tt+s}
\int_{-\infty}^0
 \frac{dw\eee^{-\frac{(w+\sqrt2(\tt-s))^2}{2(\tt-s)}}}{\sqrt{2\pi(\tt-s)}}
\left(
\int_{\sqrt 2s-w-At^\a}^{\sqrt 2s-w-Bt^\a}
\frac{dz\eee^{-\frac{z^2}{2s}}}{\sqrt{2\pi s}}\eee^{\sqrt 2\frac{\s_1}{\s_2}(w+z-\sqrt 2s)}\right)^2\nonumber\\
&&=\int_0^{\tt} ds \;\eee^{2s} 
\int_{-\infty}^0
 \frac{dw\eee^{-\sqrt 2 w  -  \frac{w^2}{2(\tt-s)}}}{\sqrt{2\pi(\tt-s)}}
 \left(
\int_{\sqrt 2s-w-At^\a}^{\sqrt 2s-w-Bt^\a}
\frac{dz\eee^{-\frac{z^2}{2s}}}{\sqrt{2\pi s}}\eee^{\sqrt 2\frac{\s_1}{\s_2}(w+z-\sqrt 2s)}\right)^2.
\eea
Now, 
\bea\Eq(square.4)&&
\int_{\sqrt 2s-w-A\t^\a}^{\sqrt 2s-w-B\t^\a}
\frac{dz\eee^{-\frac{z^2}{2s}}}{\sqrt{2\pi s}}\eee^{\sqrt 2\frac{\s_1}{\s_2}(w+z-\sqrt 2s)}
=
\int_{-At^\a}^{-Bt^\a}
\frac{dz\eee^{-\frac{(z-w+\sqrt 2s)^2}{2s}}}{\sqrt{2\pi s}}\eee^{\sqrt 2\frac{\s_1}{\s_2}z}
\nonumber\\
&&=\eee^{-s}\eee^{\sqrt 2w -\frac{w^2}{2s}}
\int_{-At^\a}^{-Bt^\a}
\frac{dz\eee^{-zw/s+\sqrt 2\left(\frac{\s_1}{\s_2}-1\right)z-\frac{z^2}{2s}}}{\sqrt{2\pi s}}.
\eea
Inserting this into \eqv(square.3), we arrive at
\be\Eq(square.5)
\int_0^{\tt}ds\int_{-\infty}^0 \frac{dw\eee^{\sqrt 2 w  -  \frac{w^2(2\tt-s)}{2s(\tt-s)}}}{\sqrt{2\pi(\tt-s)}}
\left(\int_{-At^\a}^{-Bt^\a}
\frac{dz\eee^{-zw/s+\sqrt 2t^{-\a}z-\frac{z^2}{2s}}}{\sqrt{2\pi s}}
\right)^2.
\ee
To the level of precision we care about, the integral in the square can be bounded by the maximum of its integrand, i.e. 
\bea\Eq(square.4)
\int_{-At^\a}^{-Bt^\a}
\frac{dz\eee^{-zw/s+\sqrt 2t^{-\a}z-\frac{z^2}{2s}}}{\sqrt{2\pi s}}
\leq 
\eee^{\sqrt 2 A}.
\eea
The remaining integral over $w$ is trivially bounded by $\sqrt{s/(2\tt-s)}$, which is
smaller than $1$, and we are done.

\subsection{Towards the derivative martingale}

We have seen that 
 \bea\Eq(prep.17)
 \E(\YY_k(t)\vert \FF_{t^\b})&=&\text{const.}(\sqrt{2}t^\b-x_k(t^\b) t^{-\a}(1+o(1)) \quad\mbox{and}\\
  \E(\ZZ_k(t)\vert \FF_{t^\b})&=&\text{const.}(\sqrt{2}t^\b-x_k(t^\b))(1+o(1)).
 \eea
 Moreover,  $ \E(\YY_k(t)^2\vert \FF_{t^\b})$ and $ \E(\ZZ_k(t)^2\vert \FF_{t^\b})$ and  $\E(\YY_k(t)\ZZ_k(t)\vert \FF_{t^\b})$ grow at most polynomially in $t^{\alpha}$.
 Then it follows, with $x$ as in \eqv(prep.16) and since $x_k\in \HH_{\delta}$, that
 \be\Eq(prep.18)
 \frac{\E[x^2\vert \FF_{t^\b}]}{\E[x\vert \FF_{t^\b}]}\leq
   \eee^{\sfrac{\s_1}{\s_2}(\sqrt{2}x_k(t^\b)-2t^\b)} P(t^\a)\leq C \eee^{-t^{\delta\b}}P(t^\a).
 \ee
    The right-hand side of \eqv(prep.18) converges to zero, as $t\uparrow \infty$. Using \eqv(inequality) 
    together with \eqv(prep.18), we get that the expected value in \eqv(prep.12) is equal to
   \bea\Eq(prep.19)
&& \Bigl (1-  \sfrac{\s_1}{\s_2}\CCC (\sqrt{2}t^\beta -x_k(t^\beta))\eee^{-\sqrt{2}\sfrac{\s_1}{\s_2}(\sqrt{2}t^\beta -x_k(t^\beta))} \E\left[\YY_k(t)\vert \FF_{t^\b}\right]\nonumber\\
&&\quad
 -\sfrac{\s_1}{\s_2}\CCC  \eee^{-\sqrt{2}\sfrac{\s_1}{\s_2}(\sqrt{2}t^\beta -x_k(t^\beta))} \E\left[\ZZ_k(t) \vert \FF_{t^\b}\right]\Bigr)(1+o(1)).
     \eea
    Plugging \eqv(prep.19) back into \eqv(prep.10), we get that \eqv(prep.10) is equal to
    \bea\Eq(prep.20)
   && \E\Biggl[\prod_{\stackrel{k\leq n(t^\beta)}{x_k\in \HH_{\d}}}  \Bigl (1-  \sfrac{\s_1}{\s_2}C (\sqrt{2}t^\beta -x_k(t^\beta))\eee^{-\sqrt{2}\sfrac{\s_1}{\s_2}(\sqrt{2}t^\beta -x_k(t^\beta))} 
   \E\left[\YY_k(t)\vert \FF_{t^\b}\right] (1+o(1))\nonumber\\
   &&\qquad
  -\sfrac{\s_1}{\s_2}\CCC  \eee^{-\sqrt{2}\sfrac{\s_1}{\s_2}(\sqrt{2}t^\beta -x_k(t^\beta))} \E\left[\ZZ_k(t) \vert \FF_{t^\b}\right]\Bigr)\Biggr]\nonumber\\\nonumber
 &=& \E\Big[\exp\Big(-\sum_{\stackrel{k\leq n(t^\beta)}{x_k\in \HH_{\d}}}  \Big(   \sfrac{\s_1}{\s_2}\CCC (\sqrt{2}t^\beta -x_k(t^\beta))\eee^{-\sqrt{2}\sfrac{\s_1}{\s_2}(\sqrt{2}t^\beta -x_k(t^\beta))} \E\left[\YY_k(t)\vert \FF_{t^\b}\right]\\
&&\qquad -\sfrac{\s_1}{\s_2}C  \eee^{-\sqrt{2}\sfrac{\s_1}{\s_2}(\sqrt{2}t^\beta -x_k(t^\beta))} \E\left[\ZZ_k(t) \vert \FF_{t^\b}\right]\Big)\Big)\Big](1+o(1)),
    \eea
since  $\E(x)\leq Q(t^\a)\eee^{-At^\a}$, uniformly in $x_k$, for some polynomial $Q$ as $x_k\in \GG_{t^{\b},A,B,\gamma}$.
    
 Next, we observe that, by  Lemma \thv(expectations.1),
 \bea\Eq(prep.21)
 &&\sum_{\stackrel{k\leq n(t^\beta)}{x_k\in \HH_{\d}}}     \sfrac{\s_1}{\s_2}\CCC (\sqrt{2}t^\beta -x_k(t^\beta))\eee^{-\sqrt{2}\sfrac{\s_1}{\s_2}(\sqrt{2}t^\beta -x_k(t^\beta))} \E\left[\YY_k(t)\vert \FF_{t^\b}\right]\nonumber\\
 &=& \sum_{\stackrel{k\leq n(t^\beta)}{x_k\in \HH_{\d}}}     \sfrac{\s_1}{\s_2}\CCC (\sqrt{2}t^\beta -x_k(t^\beta))^2\eee^{-\sqrt{2}\sfrac{\s_1}{\s_2}(\sqrt{2}t^\beta -x_k(t^\beta))}t^{-\alpha}\left(\frac{\eee^{-\sqrt 2y}}{\sqrt {2\pi }} +o(1)\right)\nonumber\\
 &\leq &\text{const.} t^{\beta\d-\alpha}\sum_{\stackrel{k\leq n(t^\beta)}{x_k\in \HH_{\d}}} 
 (\sqrt{2}t^\beta -x_k(t^\beta))\eee^{-\sqrt{2}\sfrac{\s_1}{\s_2}(\sqrt{2}t^\beta -x_k(t^\beta))},
 \eea
 which converges to zero in probability since  $\beta\d<\alpha$ and  the sum in the last line converges to the
 limit of the derivative martingale in probability, as will be shown in 
 Lemma \thv(almost.deriv) below.
 On the other hand, using \eqv(expectations.3) to estimate the remaining term in \eqv(prep.20), we obtain
 \bea\Eq(prep.22)
 &&\sum_{\stackrel{k\leq n(t^\beta)}{x_k\in \HH_{\d}}}     \sfrac{\s_1}{\s_2}\CCC  \eee^{-
 \sqrt{2}\sfrac{\s_1}{\s_2}(\sqrt{2}t^\beta -x_k(t^\beta))} \E\left[\ZZ_k(t)\vert \FF_{t^\b}\right]\\
 \nonumber
 &=& \sum_{\stackrel{k\leq n(t^\beta)}{x_k\in \HH_{\d}}}     \sfrac{\s_1}{\s_2}\CCC (\sqrt{2}t^\beta -
 x_k(t^\beta))\eee^{-\sqrt{2}\sfrac{\s_1}{\s_2}(\sqrt{2}t^\beta -x_k(t^\beta))} \left(\frac{\eee^{-\sqrt 
 2y}}{\sqrt {2\pi }} +o(1)\right).
 \eea
 The last expression almost looks like  (recall \eqv(derimar.1))
\be\Eq(prep.23)
\CCC Z(t^\b)\frac{\eee^{-\sqrt 2y}}{\sqrt {2\pi }}.
\ee
The next lemma asserts that this is indeed the case.

\subsection{Control of the almost martingale}
\begin{lemma}\TH(almost.deriv)  
With the notation above, 
\be\Eq(mart.1)
\sum_{k=1}^{n(t^\b)}\sfrac{\s_1}{\s_2}(\sqrt{2}t^\b-x_k(t^\b))\eee^{-\sqrt{2}\sfrac{\s_1}{\s_2}(\sqrt{2}t^\b-x_k(t^\b))}
\rightarrow Z,\quad 
\ee
in probability, as $t\uparrow \infty$, where $Z$ is the limit of the derivative martingale. Moreover,
\be\Eq(mart.2)
\sum_{k=1}^{n(t^\b)}\1_{\{x_k\in \HH_{\delta} \}}\sfrac{\s_1}{\s_2}(\sqrt{2}t^\b-x_k(t^\b))\eee^{-\sqrt{2}\sfrac{\s_1}{\s_2}(\sqrt{2}t^\b-x_k(t^\b))}\rightarrow Z,
\ee
as $t \uparrow\infty$, in probability. 
\end{lemma}
\begin{proof}
As $\s_1^2=1+t^{-\a}$ and $\s_2^2=1-t^{-\a}$, we have 
$\sfrac{\s_1}{\s_2}=1+t^{-\alpha}+o(t^{-\alpha})$,
 and to prove \eqv(mart.1) it is enough to show that
\be\Eq(mart.3)
\sum_{k=1}^{n(t^\b)}(\sqrt{2}t^\b-x_k(t^\b))\eee^{-\sqrt{2}\sfrac{\s_1}{\s_2}(\sqrt{2}t^\b-x_k(t^\b))}
\rightarrow Z, \quad t\uparrow\infty.
\ee
Next, we introduce for some $1>\gamma>1/2$
\be\Eq(mart.5)
1=\1_{x_k(t^\b)>\sqrt{2}t^\b-t^{\beta\gamma}}+\1_{x_k(t^\b)\leq\sqrt{2}t^\b-t^{\b\gamma}}.
\ee
We control the two resulting terms separately and start with 
\be\Eq(mart.8.1)
\sum_{k=1}^{n(t^\b)}\1_{x_k(t^\b)\leq\sqrt{2}t^\b-t^{\b\gamma}}(\sqrt{2}t^\b-x_k(t^\b))\eee^{-\sqrt{2}(\sqrt{2}t^\b-x_k(t^\b))}\eee^{-\sqrt{2}(\s_1/\s_2-1)(\sqrt{2}t^\b-x_k(t^\b))}.
\ee
We want to show that the term in \eqv(mart.8.1) converges to zero in probability. By the Markov inequality, it is enough to show that the expectation of \eqv(mart.8.1) converges to zero, as $t\uparrow\infty$. By the many-to-one lemma, we have
\bea\Eq(mart.8)
&&\E\Biggl[\sum_{k=1}^{n(t^\b)}\1_{x_k(t^\b)\leq\sqrt{2}t^\b-t^{\b\gamma}}(\sqrt{2}t^\b-x_k(t^\b))\eee^{-\sqrt{2}\sfrac{\s_1}{\s_2}(\sqrt{2}t^\b-x_k(t^\b))}\Biggr]\nonumber\\
&&= \eee^{t^\b} \int_{-\infty}^{\sqrt{2}t^\b-t^{\gamma \beta}} \frac{dx}{\sqrt{2\pi t^\b}}
\eee^{-\frac{x^2}{2t^\b}}
(\sqrt{2}t^\b-x)\eee^{-\sqrt{2}\sfrac{\s_1}{\s_2}(\sqrt{2}t^\b-x)}
\nonumber\\
&&\leq
\eee^{t^\b} \int_{-\infty}^{\sqrt{2}t^\b-t^{\gamma \beta}} \frac{dx}{\sqrt{2\pi t^\b}}
\eee^{-\frac{\left(x-\sqrt{2}t^\b\s_1/\s_2\right)^2}{2t^\b}   +(\s_1/\s_2-2)(\s_1/\s_2)t^\b  }
(\sqrt{2}t^\b-x) (1+o(1))
\nonumber\\
&&= \eee^{t^{\b-2\a}} \int_{-\infty}^{-t^{\gamma \beta}-\sqrt{2}t^{\b-\a}} \frac{dx}{\sqrt{2\pi t^\b}}\left(x+\sqrt{2}t^{\b-\a}\right)
\eee^{-\frac{x^2}{2t^\b}}(1+o(1))
\nonumber\\
&&\leq\text{const.}\exp\left({-\frac{(t^{\gamma \beta}+\sqrt{2}t^{\b-\a})^2}{2t^\b}}\right), 
\eea
where we computed the integral explicitly for the first summand and used Gaussian tail asymptotics for the second. The term in \eqv(mart.8) converges to zero, as $t\uparrow\infty$. 
Next, we turn to 
\bea\Eq(mart.6)
&&\sum_{k=1}^{n(t^\b)}\1_{x_k(t^\b)>\sqrt{2}t^\b-t^{\b\gamma}}(\sqrt{2}t^\b-x_k(t^\b))\eee^{-\sqrt{2}\sfrac{\s_1}{\s_2}(\sqrt{2}t^\b-x_k(t^\b))}\nonumber\\
&&> \eee^{-\sqrt{2}O(t^{-\alpha})t^{\b\gamma}}\sum_{k=1}^{n(t^\b)}\1_{x_k(t^\b)>\sqrt{2}t^\b-t^{\b\gamma}}(\sqrt{2}t^\b-x_k(t^\b))\eee^{-\sqrt{2}(\sqrt{2}t^\b-x_k(t^\b))}.
\eea
Note that the prefactor in \eqv(mart.6) converges to one, as $t^{\beta\gamma}<t^{\alpha}$. Moreover, as in  \cite{LS}, since  $x_k(t^\beta)\leq\sqrt{2}t^\beta$, a.s.,  using just that $\s_1/\s_2\geq 1$, the first line in \eqv(mart.6) is also bounded from above by
\bea\Eq(mart.6.1)
&&  \sum_{k=1}^{n(t^\b)}\1_{x_k(t^\b)>\sqrt{2}t^\b-t^{\b\gamma}}(\sqrt{2}t^\b-x_k(t^\b))\eee^{-\sqrt{2}(\sqrt{2}t^\b-x_k(t^\b))},\;\; \hbox{\rm a.s.}.
\eea
Since $\sum_{k=1}^{n(t^\b)} (\sqrt{2}t^\b-x_k(t^\b))\eee^{-\sqrt{2}(\sqrt{2}t^\b-x_k(t^\b))}$ converges to $Z$ almost surely  (see  \cite{LS}), to prove \eqv(mart.2) it is enough  to show that 
\be\Eq(mart.100)
\sum_{k=1}^{n(t^\b)}\1_{x_k(t^\b)\leq\sqrt{2}t^\b-\t^{\b\gamma}}(\sqrt{2}t^\b-x_k(t^\b))\eee^{-\sqrt{2}(\sqrt{2}t^\b-x_k(t^\b))}\to 0,
\ee
in probability as $t\uparrow \infty$.
Putting this together with \eqv(mart.6), the convergence claimed in \eqv(mart.1) follows.

To show \eqv(mart.1) we note that by \eqv(mart.8) it is enough to show that
\bea\Eq(mart.9)
&&\sum_{k=1}^{n(t^\b)}\1_{x_k(t^\b)-\sqrt{2}t^\b\not\in[-At^\d,-Bt^\gamma]} (\sqrt{2}t^\b-x_k(t^\b))\eee^{-\sqrt{2}\sfrac{\s_1}{\s_2}(\sqrt{2}t^\b-x_k(t^\b))}\rightarrow 0,
\eea
 in probability, as $t\uparrow\infty$.
Note that by the same upper and lower bounds as in \eqv(mart.6) and \eqv(mart.6.1),
to prove   \eqv(mart.9) and \eqv(mart.100), it is enough to show that
\bea\Eq(mart.101)
&&\sum_{k=1}^{n(t^\b)}\1_{x_k(t^\b)-\sqrt{2}t^\b\not\in[-t^\d,-t^\gamma]} (\sqrt{2}t^\b-x_k(t^\b))\eee^{-\sqrt{2} (\sqrt{2}t^\b-x_k(t^\b))} \nonumber\\
&&=\sum_{k=1}^{n(t^\b)}\left(\1_{x_k(t^\b)-\sqrt{2}t^\b<-t^\gamma} +\1_{x_k(t^\b)-\sqrt{2}t^\b>-t^\d }\right)(\sqrt{2}t^\b-x_k(t^\b))\eee^{-\sqrt{2} (\sqrt{2}t^\b-x_k(t^\b))}
 \nonumber\\
&&\equiv (I)+(II)
\eea
converges to zero in probability as $t\uparrow \infty$.
Following the computation in \eqv(mart.8), we get, using the many-to-lemma, that the expectation of (I) in \eqv(mart.101) is equal to
\bea\Eq(mart.105)
&&\eee^{t^\beta}\int^{\sqrt{2}t^\beta-t^{\b\g}}_{-\infty}\frac{dx}{\sqrt{2\pi t^\b}}\eee^{-\frac{x^2}{2t^\b}} (\sqrt{2}t^\b-x)\eee^{-\sqrt{2}(\sqrt{2}t^\b-x)}
\nonumber\\
&&\qquad=\int^{\infty}_{t^{\b\g}}\frac{dy}{\sqrt{2\pi t^\b}}y\eee^{-\frac{y^2}{2t^\b}} 
=\frac{t^{\b/2}}{\sqrt {2\pi}}\eee^{-t^{\b(2\g-1)}/2}.
\eea
 As $\gamma>1/2$ the \eqv(mart.105) converges to zero as $t\uparrow \infty$.
 For (II) in \eqv(mart.101), we have that, for $r$ large enough,   
\be\Eq(mart.103)
\P\left((II)>\e\right)\leq \P\left(\{(II)>\e\} \land \{\forall_{k\leq n(t^\b)}x_k\in\TT_{r,t^\b}\}\right) +\epsilon.
\ee
Using again the Markov inequality and the   many-to-one lemma, 
we obtain the bound
 \bea\Eq(mart.102)
 && \P\left(\{(II)>\e\} \land\{ \forall_{k\leq n(t^\b)}x_k\in\TT_{r,t^\b}\}\right)\nonumber\\
&&\leq\E\Biggl[\sum_{k=1}^{n(t^\b)}\1_{x_k(t^\b)-\sqrt{2}t^\b>-t^\d}\1_{\forall_{r\leq s\leq t^\b} x_k(s)\leq \sqrt{2}s} (\sqrt{2}t^\b-x_k(t^\b))\eee^{-\sqrt{2} (\sqrt{2}t^\b-x_k(t^\b))} \Biggr]\nonumber\\
&&= \eee^{t^\beta}\int^{\sqrt{2}t^\beta}_{\sqrt{2}t^\beta-t^{\b\d}}\frac{dx}{\sqrt{2\pi t^\b}}\eee^{-\frac{x^2}{2t^\b}} (\sqrt{2}t^\b-x)\eee^{-\sqrt{2}(\sqrt{2}t^\b-x)}\sqrt{\frac{2}{\pi}}\frac{(\sqrt{2}t^\b-x)\sqrt{r}}{t^\b},
\eea
using that the Brownian bridge is independent from its endpoint and \eqv(p.4), with $t/2$ replaced by $t^\beta$.
The integral in \eqv(mart.102) is computed as in \eqv(mart.105) and 
we see that \eqv(mart.102) is equal to
\be\Eq(mart.104) 
 \sqrt{\frac{2r}{\pi}}\int^{t^{\b(\d-1/2)}}_0 \frac{dy}{\sqrt{2\pi}}y^2\eee^{-y^2},
\ee
which converges,  for any $r$ fixed, to zero as $t\uparrow \infty$, since  $\d<1/2$. Putting the estimates in \eqv(mart.105) and \eqv(mart.104) together, we  obtain that \eqv(mart.101) converges to zero in probability as $t\uparrow \infty$. This concludes the proof of Lemma \thv(almost.deriv).
\end{proof}
\subsection{Conclusion of the proof}
Using Lemma \thv(almost.deriv) we see that indeed the right-hand side of \eqv(prep.22) converges, as first $t\uparrow\infty$ and then $A\downarrow0$ and $B\uparrow\infty$, 
in probability to 
\be\Eq(prep.23.1)
\CCC Z\frac{\eee^{-\sqrt 2y}}{\sqrt {2\pi }}.
\ee
Together with the fact that the term in \eqv(prep.21) converges to zero, we get that 
\eqv(prep.10) converges to 
\be\Eq(almost.1)
\E \left[\eee^{-\frac {2^{3/2}}{\sqrt{2\pi} }CZ\eee^{-\sqrt 2 y}}\right],
\ee 
which implies \eqv(M.Th2) and proves Theorem \thv(main.1) in the case when $\s_1^2=1+t^{-\a}$.

 \section{The law of the maximum: The case $\s_1^2=1-t^{-\a}$}

We now consider the case when $\s_1^2=1-t^{-\a}$ and $\s_2^2=1+t^{-\a}$. In this case, we have 
\be\Eq(below.1)
m(t)=\sqrt 2 t -\frac {3-\g}{2\sqrt 2}\ln t, 
\ee
where $\g=2-4\a$, as long as $\a\leq 1/2$.
The aim of this section is to prove that
\begin{theorem}\TH(B.Th1)
 Let $m(t)$ be as in \eqv(below.1). Then 
\be\Eq(B.Th2)
\lim_{t\to\infty}\P\left[\max_{1\leq k \leq n(t)}\tilde x_k(t)-m(t)\leq y\right]= \E\left[\exp\left(-CZ\eee^{-\sqrt{2}y}\right)\right].
\ee
$Z$ is  the limit of the derivative  martingale 
and $C$ is a positive constant.
\end{theorem}

The structure of the proof is identical to that in the previous section. 

\subsection{Localisation of paths}

To prove Theorem \thv(B.Th1) we need control the position of particles until time $t/2$. 
Only the position of the path at time $t/2$ needs to be modified from the previous section, i.e we redefine 
\be\Eq(B.11)
\GG_{s,A}=\bigl\{X\big \vert X(s)-\sqrt{2}s\s_1\in  [-A  s^{1/2},As^{1/2}]
\bigr\}.
\ee

\begin{proposition}\TH(B.Prop1) Let $\s_1^2=1-t^{-\a}, \s_2^2=1+t^{-\a}$ with $\a\in(0,1/2)$.
 For any $d\in \R$ and 
any  $\e>0$,  there exists a constant $A>0$ such that, for all $t$ large enough,
\be\Eq(b.p.1)
\P\left[\exists_{j\leq n(t)}: \{\tilde x_j(t)>m(t)-d\} \land \{\s_1^{-1}\tilde   x_j(t/2)
\not\in \GG_{t/2,A}\}
\right]\leq \e.
\ee
\end{proposition}

\begin{proof} 
Note that $\s_1\sqrt 2t/2 =\sqrt 2 t/2- \sqrt 2 t^{1-\a}/4 +O(t^{1-2\a})$.
Abbreviate 
\be
I\equiv [\sqrt 2t/2-\t-A\sqrt t,\sqrt 2t/2-\t+A\sqrt t],
\ee
 with 
 $\t=\sqrt 2(1-\s_1)t/2$. Note that 
 $ \t= \sqrt 2t^{1-\a}/4 +O(t^{1-2\a})\gg\sqrt t$.  The probability in question can be written in the form
\be\Eq(b.2)
\P\left(\exists_{k\leq n(t/2)}  : \{\s_1x_k(t/2)>  m(t)-\s_2\max_{\ell\leq n^k(t/2)}x_\ell^k-d\} \land 
\{x_k(t/2) \not \in  I\}\right).
\ee
We can also insert the condition that particles stay below the line $\sqrt 2 s$ for all time at no cost. 
Then the 
expression in \eqv(b.2) becomes
\bea\Eq(b.2.1)\nonumber
&&\P\Bigl(\exists {k\leq n(t/2)}  :  \{\s_2\max_{\ell\leq n^k(t/2)}x_\ell^k >  m(t)- \s_1x_k(t/2)-d \}\land
\{x_k(t/2)\not \in I\} \\
&&\quad\quad\quad\land \{ x_k(s)\leq \sqrt 2s, \forall s\in [r,t/2]\}
\Bigr),
\eea
(where it is understood that $r\uparrow\infty$ after $t\uparrow\infty$). 
By the many-to-one lemma, 
this is bounded from above by
\bea\Eq(b.3)
    &&\eee^{t/2}\E\left[ \1_{\s_1x_1(t/2)\not\in I}
    \1_{x_k(s)\leq \sqrt 2s, \forall s\in [r,t/2]}
    \1_{ \max_{\ell\leq n^k(t/2)} \s_2x_\ell^k(t/2) >  m(t)- \s_1x_k(t/2)-d }\right]\\\nonumber
    &&=\eee^{t/2}\int_{I^c} \frac{  \eee^{-\frac{z^2}{t}}}{\sqrt {\pi t}}
    \P\left(\zet^{t/2}_{0,\sqrt 2 t/2-z}(s)\leq 0, \forall_{s\in (r,t/2]}\right)
    \P\left( \s_2\max_{\ell\leq n^k(t/2)}x_\ell^k(t/2) >  m(t)- \s_1z-d \right)dz,
    \eea
where $\zet^{t/2}_{0, y}$ denotes the Brownian bridge from $0$ to $y$ in time $t/2$
and we wrote $I^c$ short for $ I^c\cap (-\infty,\sqrt 2t/2]$.
The probability regarding the Brownian bridge satisfies, since $\t\gg \sqrt t$, 
\be\Eq(b.p.4)
    \P\left(\zet^{t/2}_{0,\sqrt 2 t/2-z}(s)\leq 0, \forall_{s\in (r,t/2]}\right)
    \leq  \frac {\sqrt r}  {\sqrt \pi t^\a}.
    \ee
    We now write the integral as (set $J_A=(-A\sqrt t,A\sqrt t)$),
 \bea\Eq(b.p.5)\nonumber
          &&\eee^{t/2}\int_{J_A^c} \frac{  \eee^{-\frac{(\sqrt2 \s_1t/2-y)^2}{t}}}{\sqrt {\pi t}}
   \P\left(\zet^{t/2}_{0,-y}(s)\leq 0, \forall_{s\in (r,t/2]}\right)\\
    &&\quad\quad\times
    \P\left(  \s_2\max_{\ell\leq n^k(t/2)}x_\ell^k(t/2) >  m(t)-\s_1^2\sqrt 2t/2 +\s_1 y-d \right)dy.
    \eea
    A simple calculation shows that 
    \bea\Eq(mmm.1)
    \frac{m(t)- \s_1^2\sqrt 2 t/2}{\s_2} &=& \sqrt 2t/(2\s_2) +\sqrt 2(\s_2 -\s_2^{-1})t/2  -\frac {3-\g}{2\sqrt 2\s_2} \ln t
    \nonumber\\&=& \sqrt 2t/2   +\sqrt 2 t(\s_2-1)/2-\frac {3-\g}{2\sqrt 2\s_2} \ln t.
    \eea
    Hence the probability involving the maximum in \eqv(b.p.5) reads
    \be\Eq(b.p.6)
    \P\left( \max_{\ell\leq n^k(t/2)}x_\ell^k(t/2)-\sqrt 2t/2 > 
    \sqrt 2 t(\s_2-1)/2-\frac {3-\g}{2\sqrt 2\s_2} \ln t+\frac{\s_1}{\s_2}y-d/\s_2 \right)dy.
    \ee
    Using Proposition  \thv(A.Prop1), respectively \eqv(probab.1),
    we see that this probability equals, asymptotically as $t\uparrow\infty$, to 
    \bea\Eq(b.6)
   && Ct^{-3/2}\left( \sqrt 2 t(\s_2-1)/2-\frac {3-\g}{2\sqrt 2\s_2} \ln t+\frac{\s_1}{\s_2}y-d\right)
    \eee^{-\sqrt 2\left(    \sqrt 2 t(\s_2-1)/2-\frac {3-\g}{2\sqrt 2\s_2} \ln t+\frac{\s_1}{\s_2}y-d \right)}\nonumber\\&&\quad\times
   \eee^{-\left(\sqrt 2 t(\s_2-1)/2-\frac {3-\g}{2\sqrt 2\s_2} \ln t+\frac{\s_1}{\s_2}y-d \right)^2/t}.
   \eea
    The terms in the last exponential can be written as 
    \be\Eq(sss.1)
    (\s_2-1)^2t/2 +\sqrt 2\frac {\s_1}{\s_2}(\s_2-1) y +\frac{\s_1^2}{\s_2^2}y^2/t +o(1).
    \ee
    
       Inserting this and the bound \eqv(b.p.4) into \eqv(b.p.5), we see that this term is not larger than 
       \be\Eq(b.p.7)
     \int_{|y|>A\sqrt t}
  \frac{  \eee^{- \left(1+{\s_1^2}/{\s_2^2}\right)y^2/t}}{\sqrt {\pi t}}
   \frac {\sqrt r} {\sqrt {\pi }t^\a}
    C  \t t^{-\g/2}
  \eee^{\sqrt 2 d}
 dy,
    \ee
   Recalling that $\g=2-4\a$, this becomes
         \be\Eq(b.p.7)
   C    \frac {\sqrt r} {\sqrt {\pi }}
  \eee^{\sqrt 2 d} \int_{|y|>A}
  \frac{  \eee^{- \left(1+{\s_1^2}/{\s_2^2}\right)y^2}}{\sqrt {\pi }}
 dy,
    \ee
  For any finite $r$, this tends to zero, as $A\uparrow \infty$. 
   This concludes the proof of the proposition.
   \end{proof}

\subsection{Recursive structure}
As in the previous section,  and with the same notation, we write
\be\Eq(B.prep.1)
 \P\left[\max_{1\leq k \leq n(t)}\tilde x_k(t)-m(t)> y\right]= \P\left[\max_{ k \leq n(t/2), l\leq n^k(t/2)}  \sigma_1 x_k(t/2)+\sigma_2 x^k_l(t/2)-m(t)> y\right].
\ee
We again need to define  shifted versions of the event $\GG$ and $\TT$ by
\be\Eq(B.11.2)
\GG_{s,A,S,T}=\bigl\{X\big \vert X(s)-\sqrt{2}s\s_1+S\in  [-A  (s+T)^{1/2},A(s+T)^{1/2}]
\ee
 By Propositions \thv(B.Prop1), \thv(P.Prop0), and \thv(P.Prop2), we have that
\bea\Eq(B.prep.5)
&&\P\left[\max_{1\leq k \leq n(t)}\tilde x_k(t)-m(t)> y\right]\\\nonumber&&
=\P\left[\exists_{ k \leq n(t/2), l\leq n^k(t/2)} :\{ \sigma_1 x_k(t/2)+\sigma_2 x^k_l(t/2)-m(t)> y\}
\land \{x_k\in 
\LL_{t^\b,t/2,A}\}\right]+O(\e).
\eea
where 
\be\Eq(B.prep.6)
\LL_{t^\b,t/2,A}=\GG_{t/2,A}\cap\HH_{\d}\cap \TT_{t^{\b},t/2}.
\ee
In view of Lemma \thv(minimax.1), it is enough to analyse the probability in the second line of 
\eqv(B.prep.5), which, 
as in Eq. \eqv(prep.2),  can be written as
\be\Eq(B.prep.2)
1-\E\Biggl[\prod_{\stackrel{k\leq n(t/2)}{x_k\in \LL_{t^\b,t/2,A}}}\left(1-\P\left[\max_{l\leq n^k(t/2)}x^k_l(t/2)> \frac{m(t)+y-\sigma_1 x_k(t/2)}{\sigma_2}\big\vert \FF_{t/2}\right]\right)\Biggr].
\ee
Since  $x_k\in \GG_{t/2,A}$, we can use the tail asymptotics given in Proposition \thv(A.Prop1) to control the 
conditional probability in \eqv(B.prep.2)\footnote{Note that this is true if $\a>0$. Otherwise, we cannot use 
the tail asymptotics and thus in the case $\s_1^2-1=O(1)$, the behaviour changes completely, see \cite{BH14.1}.}.
 Namely,
\be\Eq(B.prep.7)
 \P\left[\max_{l\leq n^k(t/2)}x^k_l(t/2)> \frac{m(t)+y-\sigma_1 x_k(t/2)}{\sigma_2}\big\vert \FF_{t/2}\right]
 =2^{3/2} C \Gamma_k(t)\eee^{-\sqrt{2}\Gamma_k(t)-\G_k(t)^2/t}  (1+o(1)),
\ee
where the $o(1)$ error term is uniform in the range of possible values for $x_k(t/2)$ as $x_k\in \GG_{t/2,A}$ and 
\be\Eq(B.prep.8)
\Gamma_k(t)=\frac{m(t)+y-\sigma_1 x_k(t/2)}{\sigma_2}-\left(t/\sqrt{2}-\frac{3}{2\sqrt{2}}\log(t)\right).
\ee
Plugging \eqv(B.prep.7) back into \eqv(B.prep.2) we obtain that the expectation in \eqv(B.prep.2) is equal to
\bea\Eq(B.prep.9)
\E\Biggl[\prod_{\stackrel{k\leq n(t/2)}{x_k\in \LL_{t^\b,t/2,A}}} \exp\left( -2^{3/2}C \Gamma_k(t)\eee^{-\sqrt{2}\Gamma_k(t)-\Gamma_k(t)^2/t}\right)\Biggr](1+o(1)),
\eea
 since $\Gamma_k(t)>At^{1/2}$  for $x_k\in \GG_{t/2,A} $. Next, we rewrite the expectation in  \eqv(B.prep.9) by conditioning on $\FF_{t^\beta}$ as
 \be\Eq(B.prep.10)
 \E\Biggl[\prod_{\stackrel{k\leq n(t^\beta)}{x_k\in \HH_{\d}}} 
 \E\Biggl[\prod_{\stackrel{j\leq n^k(t/2-t^\beta)}{x_j^k\in \GG_{t/2-t^\b,A,x_k(t^\beta)-\sqrt{2}t^\beta,t^{\beta}}\cap\TT_{0,t/2-t^\beta,x_k(t^{\beta})-\sqrt{2}t^\beta}}} \hspace{-15mm}
 \exp\left(- C \D_k(t)\eee^{-\sqrt{2}\D_k(t)-\D_k(t)^2/t}\right)\big \vert \FF_{t^\beta}\Biggr]\Biggr],
 \ee
with $\GG_{t/2-t^\b,A,x_k(t^\beta)-\sqrt{2}t^\beta.t^\beta}$  as defined in \eqv(B.11.2) and
\bea\Eq(B.prep.11)
\D_k(t)&=&\frac{m(t)+y-\sigma_1 \left(x_k(t^\b)+x_j^k(\tt)\right)}{\sigma_2}-\left(t/\sqrt{2}-\frac{3}{2\sqrt{2}}\log(t)\right)\\\nonumber
&=&\sqrt 2 t (\s_2-1)/2 +\frac {2-4\a}{2\sqrt 2}\ln (t)
+\frac{\s_1}{\s_2} \left(\sqrt 2t^*\s_1-x_j^k(t^*)+\sqrt 2t^\b\s_1-x_k(t^\b)\right)\\&&\nonumber
+y/\s_2+O(t^{-\a}\ln t)),
\eea
where, for each $k$, $(x^k_j(\cdot))_{l\leq n^k(t/2-t^\b)}$ are particles of an independent standard branching Brownian motion.  
Note also that, taking into account the localisation,  that 
\be\Eq(rough-delta.1)
\D_k(t)=\sqrt 2t(\s_2-1)/2+O(\sqrt t)=\sqrt 2 t^{1-\a}/4+O(\sqrt t).
\ee
Thus, in the prefactor of the exponential, we can replace $\D_k(t)$ simply by $\sqrt 2 t^{1-\a}/4$.
Next, using again localisation, 
\bea\Eq(B.prep.11.1)
\D_k(t)^2/t&=&\frac 12 t(\s_2-1)^2 
+\left(\frac{\s_1}{\s_2} \left(\sqrt 2t^*\s_1-x_j^k(t^*)\right)\right)^2/t
\nonumber\\&&+\frac{\s_1}{\s_2} \left(\sqrt 2t^*\s_1-x_j^k(t^*)\right)\sqrt 2 (\s_2-1)
+o(1).
\eea
Putting both terms together, we get for the terms in the exponent,
\bea\Eq(b.allexp.1)\nonumber
&&-\frac t2 (\s_2^2-1)-\frac{4\a -2}{2}\ln(\ti)+
\sqrt 2\frac{\s_1}{\s_2} \left(\sqrt 2t^*\s_1-x_j^k(t^*)+\sqrt 2t^\b\s_2-x_k(t^\b)\right)
\nonumber\\&&-\left(\frac{\s_1}{\s_2} \left(\sqrt 2t^*\s_1-x_j^k(t^*)\right)\right)^2/t
-\frac{\s_1}{\s_2} \left(\sqrt 2t^*\s_1-x_j^k(t^*)\right)\sqrt 2 (\s_2-1)
\nonumber\\&&-\sqrt 2 y +o(1).
\eea
We set
\be\Eq(B.prep.15)
\wt\LL_{t^\beta,t/2-t^\b,x_k(t^\beta)}\equiv\GG_{t/2-t^\b,A,x_k(t^\beta)-\sqrt{2}t^\beta,t^\beta} \cap\TT_{0,t/2-t^\beta,x_k(t^\b)-\sqrt{2}t^\b}
\ee
We rewrite the inner expectation in \eqv(B.prep.10) as
\bea\Eq(B.prep.12)
&&\E\Biggl[ \exp\Biggl(-\sum_{\stackrel{j\leq n^k(t/2-t^\beta)}{x_j^k\in \wt\LL_{t^\beta,t/2-t^\b,x_k(t^\beta)}}} 
2^{3/2} C \D_k(t)\eee^{-\sqrt{2}\D_k(t)-\D_k(t)^2/t}\Biggr)\big \vert \FF_{t^\beta}\Biggr]\nonumber\\
 &&=
 \E\left[ \exp\left(-   {C}  t^\a
 \eee^{-\frac t2\left(\s_2^2-1\right) -\sqrt 2y}
  \eee^{-\frac{\s_1}{\s_2} 
  \left(\sqrt 2t^\b\s_1-x_k(t^\b)\right)}
 \YY_k(t)\right)
  \big \vert \FF_{t^\beta}\right],
\eea
where
\be\Eq(B.prep.13)
\YY_k(t)=\sum_{\stackrel{j\leq n^k(t/2-t^\beta)}{x_j^k\in \wt\LL_{t^\beta,t/2-t^\b,x_k(t^\beta)}}}
\eee^{-\sqrt{2}\frac{\sigma_1}{\sigma_2} \left(\sqrt{2} \tt\s_1-x^k_j(\tt)\right)}
\eee^{
-\sqrt 2\frac{\s_1}{\s_2} \left(\sqrt 2t^*\s_1-x_j^k(t^*)\right)(\s_2-1)
-\left(\frac{\s_1}{\s_2} \left(\sqrt 2t^*\s_1-x_j^k(t^*)\right)\right)^2/t} .
\ee
Here  we used that $t^{1-\a}t^{-\g/2}=t^\a$. Note that this time, there is no term involving $\ZZ_k$!
As in the case $\s_1>\s_2$, we can effectively replace  in \eqv(B.prep.12)  $\exp(\cdot)$ by $1 +(\cdot)$, 
compute the conditional expectation, and return then to $\exp(\cdot)$. This gives,
\bea\Eq(b.allexp.3)
&& \E\left[ \exp\left(-   {C}  t^\a
 \eee^{-\frac t2\left(\s_2^2-1\right) -\sqrt 2y}
  \eee^{-\frac{\s_1}{\s_2} 
  \left(\sqrt 2t^\b\s_1-x_k(t^\b)\right)}
 \YY_k(t)\right)
  \big \vert \FF_{t^\beta}\right]\nonumber\\
  &&
= \exp\left(-   {C}  t^\a
 \eee^{-\frac t2\left(\s_2^2-1\right) +\sqrt 2y}
  \eee^{-\frac{\s_1}{\s_2} 
  \left(\sqrt 2t^\b\s_1-x_k(t^\b)\right)}
 \E[\YY_k(t)|\FF_{t^\b}]\right)(1+o(1)). 
 \eea
 The proof of \eqv(b.allexp.3) is completely analogous to the corresponding result in the case 
 $\s_1>1$ and will be skipped.

\subsection{Computation of the main term}

We now come to the computation of the averages of $\YY_k(t)$.

\begin{lemma}
\TH(expectations.4)
With the notation from the last subsection, 
\be\Eq(expectations.5)
 \E\left[\YY_k(t)|\FF_{t^\b}\right]=\s_2 (\sqrt {2}t^\b-x_k(t^\b)) t^{-\a} \eee^{(1-\s_1^2)\tt} /\sqrt 2
 \left(1+o(1)\right),\ee
 where $o(1)$ tends to zero as first $t\uparrow \infty$ and then $A\uparrow \infty$.
 \end{lemma}
 
 \begin{proof}

 Since the $x_k(t^\b)$ must be of order $t^{\b/2}$ below $\sqrt 2 t^\b$, 
 the bridges involved must go from $x_k(t^\b)$ to its endpoint $x_k(t)$ and stay below $\sqrt 2 s$ all the time. This condition produces a factor \footnote{Note that this holds only if $\a>0$. As soon as $1-\s_1^2 =O(1)$, the bridge 
 condition disappears completely. This is why in that case the McKean martingale appears instead of the derivative martingale.}
 \be\Eq(simpler.1)
2 \frac {(\sqrt {2}t^\b-x_k(t^\b))(\sqrt2t/2-x_k(\tt))}\tt = {\sqrt 2} (\sqrt {2}t^\b-x_k(t^\b))t^{-\a}(1+o(1)).
 \ee
Note that the constraint on the endpoint of $x^k_j(t^*)$ is that 
$x_k(t^\b)+x^k_j(t^*) -\sqrt 2t\s_1/2\in (-A\sqrt t,A\sqrt t)$, but since 
$|x_k(t^\b)|$ is at most of order $t^\b\ll \sqrt t$, this constraint is equivalent to 
$x^k_j(t^*) -\sqrt 2t^*\s_1\in (-A\sqrt t,A\sqrt t)$. 
 Thus 
 \bea
 \Eq(b.mckean.1)
 &&
 \E\left[\YY_k(t)|\FF_{t^\b}\right]= \sqrt 2 \eee^{\tt} (\sqrt {2}t^\b-x_k(t^\b))t^{-\a}\\\nonumber
 &&\times   \int_{\sqrt 2 t^*\s_1 -A\sqrt t}^{\sqrt2 t^*\s_1 +A\sqrt t}
\frac{dz\eee^{-\frac {z^2}{2\tt}}}{\sqrt{2\pi\tt}}
 \eee^{\sqrt 2 \frac{\s_1}{\s_2}(z-\sqrt 2\tt\s_1)}
 \eee^{
-\sqrt2\frac{\s_1}{\s_2} \left(\sqrt 2t^*\s_1-z\right) (\s_2-1)
-\left(\frac{\s_1}{\s_2} \left(\sqrt 2t^*\s_1-z\right)\right)^2/t}.
 \eea
 Shifting the integration variable, the integral in the last expression becomes
 \bea\Eq(b.mckean.2)
&&    \int_{-A\sqrt t}^{A\sqrt t}
\frac{\eee^{-\frac {(z+\sqrt 2\tt\s_1)^2}{2\tt}}}{\sqrt{2\pi\tt}}
 \eee^{\sqrt 2 \frac{\s_1}{\s_2}z}
 \eee^{\sqrt 2
\frac{\s_1}{\s_2} z (\s_2-1)
-\frac{\s_1^2}{\s_2^2} z^2/t}dz\nonumber\\
&&=
\eee^{-\s_1^2 \tt}  \int_{-A\sqrt t}^{A\sqrt t}
\frac{\eee^{-  \frac {z^2}{\s_2^2\tt}}}{\sqrt{2\pi\tt}}
 dz(1+o(1))\to \eee^{-\s_1^2 \tt} \s_2/\sqrt 2, \quad \text{as} \;\;A\uparrow \infty.
 \eea
 This implies \eqv(expectations.5) and concludes the proof of the lemma.
\end{proof}

\subsection{Towards the derivative martingale}

Inserting \eqv(expectations.5) into \eqv(b.allexp.3), 
we see that this now becomes
\be\Eq(b.allexp.4)
 \exp\left(-   
  {C}  \eee^{ -\sqrt 2y}(\sqrt 2t^\b-x_k(t^\b))
  \eee^{-\frac{\s_1}{\s_2} 
  \left(\sqrt 2t^\b\s_1-x_k(t^\b)\right)}
 \right)\left(1+o(1)\right).
 \ee
 Plugging this into \eqv(B.prep.10), this becomes 
 \be\Eq(b.allexp.5)
\E\Biggl[ \exp\Biggl(-\sum_{\stackrel{k\leq n(t^\beta)}{x_k\in \HH_{\d}}}     \sfrac{\s_1}{\s_2}C (\sqrt{2}t^\beta -x_k(t^\beta))\eee^{-\sqrt{2}\sfrac{\s_1}{\s_2}(\sqrt{2}t^\beta -x_k(t^\beta))} \eee^{-\sqrt 2y}\Biggr)\Biggr]
 (1+o(1)).
 \ee
It remains to show that the sum in the exponential converges to the limit of the derivative martingale:

\begin{lemma}\TH(b.allexp.6)
With the notation above, 
\be\Eq(b.mart.2)
\sum_{k=1}^{n(t^\b)}\1_{x_k\in \HH_{\delta} }\sfrac{\s_1}{\s_2}(\sqrt{2}t^\b-x_k(t^\b))\eee^{-\sqrt{2}\sfrac{\s_1}{\s_2}(\sqrt{2}t^\b-x_k(t^\b))}\rightarrow Z,
\ee
in probability, as $t\uparrow \infty$, where $Z$ is the limit of the derivative martingale. 
\end{lemma}
\begin{proof} The proof of this lemma is completely analogous to that of Lemma \thv(almost.deriv) and will be skipped.
\end{proof}

Form this the proof of Theorem \thv(main.1) follows in the case $\s_1<1$.

\section{The Laplace functional. Proof of Theorem \thv(main.2).}

To control the extremal processes, we need to analyse the Laplace functionals. It will in fact be enough to consider functions $\phi:\R\rightarrow \R_+$ of the form 
\be\Eq(laplace.1)
\phi(x)=\sum_{\ell=1}^Lc_\ell \1_{x\geq u_\ell},
\ee
with $L\in\N$, $c_\ell >0$, and $u_\ell\in \R$ (see \cite {BovHar13, bbm-book}).
We need to compute 
\bea\Eq(laplace.2)
\Psi_t(\phi)&\equiv& \E\left[\eee^{-\int \phi(x)\EE_t(dx)}\right]
\\\nonumber
&=& \E\left[\eee^{-\sum_{k=1}^{n(t)}\phi(\tilde x_k(t)-m(t))}\right]
\\\nonumber
&=& \E\left[\eee^{-\sum_{k=1}^{n(t/2)}\sum_{j=1}^{n^k(t/2)}\phi(\s_1  x_k(t/2)+\s_2 x^k_j(t/2)-m(t))}\right]
\\\nonumber
&=& \E\left[ \prod_{k=1}^{n(t/2)}
\E\left[\eee^{-\sum_{j=1}^{n^k(t/2)}\phi(\s_1  x_k(t/2)+\s_2 x^k_j(t/2)-m(t))}\big|\FF_{t/2}  \right]\right].
\eea
As in the previous chapters, we would like to interpret the conditional expectation in the product as
a solution of the F-KPP equation and use the asymptotics of these solutions. However, there is a small 
problem due to the fact that the $\s_2$ that multiplies $x^k_j(t/2)$ depends on $t$. 
We will see that this problem can be solved rather easily with the help of the maximum principle.

To see this, consider, for fixed $t\in \R$ and $f:\R\rightarrow\R_+$,
\bea\Eq(laplace.3)
&&\E\left[\prod_{j=1}^{n(s)} f\left(\s_1(t)x(t)-\s_2(t)x_j(s)\right)\right]
=\E\left[\prod_{j=1}^{n(s)} f^t\left(\sfrac{\s_1(t)}{\s_2(t)}x(t)-x_j(s)\right)\right]\equiv v^t\left(s, \sfrac{\s_1(t)}{\s_2(t)}x(t)\right),\nonumber\\
\eea
where $f^t(x)=f(x\s_2(t))$. 
Then, for fixed $t$, $1-v^t$ is a solution of the F-KPP equation with initial condition 
$1-v^t(0,x)=1-f^t(x)$.  Provided that $f$ (and $f^t$) satisfies the 
assumptions of Bramson's theorem, we can derive the large-$s$ asymptotics for $v^t$. 
However, we want to look at the asymptotics when $s=t/2$ and $t\uparrow\infty$. Since in our cases, 
$f^t(x)\rightarrow f(x)$, as $t\uparrow\infty$, the initial conditions satisfy Bramson's conditions 
uniformly in $t$ and bounds on $v^t(s,x)$ for large $s$ hold uniformly in $t$. 

Fortunately, the maximum principle allows to overcome this difficulty.

\begin{lemma} 
\TH(laplace.4)
Assume that $f^t$ is such that for all $t>t_0$, and all $x\geq 0$, 
\be\Eq(laplace.5)
f(x)\leq f^t(x)\leq f^{t_0}(x).
\ee
Then, for all $x\geq 0$, and all $t>t_0$, 
\be\Eq(laplace.6)
v(s,x)\leq v^t(s,x)\leq v^{t_0}(s,x).
\ee
In particular, 
\be\Eq(laplace.7)
v(t,x(t))\leq v^t(t,x(t))\leq v^{t_0}(t,x(t)).
\ee
The same holds if all inequalities are reversed.
\end{lemma}

\begin{proof}
The proof is straightforward from the maximum principle, see Proposition 3.1 in \cite{B_C} resp. Proposition 6.4 in \cite{bbm-book}.
\end{proof}

With this information in mind we get the following slight generalisation of Proposition \thv(A.Prop1).

\begin{proposition}\TH(C.Prop1)
Let $u^t$ be a family of solutions to the F-KPP equation with initial data satisfying
\be \Eq(pointwise.1) 
u^t(0,x)\rightarrow u(0,x),
\ee
pointwise and monotone for $x\geq 0$, as $t\uparrow \infty$, where $u(0,x)$ satisfies
\begin{itemize}
\item[(i)] $0\leq u(0,x)\leq 1$;
\item[(ii)] for some $h>0$, $\limsup_{t\to\infty}\frac{1}{t}\log\int_t^{t(1+h)}u(0,y)\mbox{d}y\leq-\sqrt{2}$;
\item[(iii)] for some $v>0,\,M>0$, and $N>0$, it holds that $\int_{x}^{x+N}u(0,y)\mbox{d}y>v$, for all $x\leq -M$;
\item[(iv)] moreover, $\int_0^\infty u(0,y)y\eee^{2y}\mbox{d}y<\infty$.
\end{itemize}
Then we have, for $0<x=x(t)$ such that $\lim_{t\uparrow\infty} x(t)/t=0$ 
\be\Eq(tail.1.1)
\lim_{t\to\infty}e^{\sqrt{2}x}\eee^{x^2/2t}x^{-1}u(t,x+\sqrt{2}t-\sfrac 3{2\sqrt 2}\ln t),
\ee
where $C$ is a strictly positive constant that depends only on
 the initial condition $u(0,\cdot)$. 
 More precisely, 
 \be\Eq(mopre.1.1)
 C\equiv \lim_{r\uparrow\infty} \sqrt {\frac 2\pi} \int_0^\infty u(r,y+\sqrt 2r) \eee^{\sqrt 2y}ydy,
 \ee
 where $u$ is the solution of the F-KPP equation with initial condition $u(0,x)$.
\end{proposition}

\begin{proof} The proof is essentially a rerun of the proofs in the case of fixed initial condition
(see e.g. the proofs of Propositions 7.1 and  9.8 in \cite{bbm-book}. The main point is to control the limit of expressions of the type 
\be\Eq(neat.1)
\int_0^\infty v^t(r,y,\sqrt 2r) \eee^{\sqrt 2 y-\frac {(x(t)-y)^2}{2(t-r)}}
\left(1-\eee^{-2y \frac {x(t)}{t-r}}\right)dy.
\ee
The idea is always to take the pointwise limit in the integral as $t\uparrow
\infty$ and justify this by showing that the hypothesis of Lebesgue's dominated convergence theorem
are satisfied. The only new aspect here is that we also want to replace $v^t$ by its limit $v$. 
But this is precisely justified due to the maximum principle.
\end{proof}

Having established the tail asymptotics of the  solution, the remainder of the analysis of the Laplace 
functional  is now exactly 
the same as that of the law of the maximum in the preceding sections. This proves Theorem \thv(main.2).

\end{document}